\documentclass[a4paper, 12pt]{amsart}

\usepackage{amsmath}
\usepackage{amssymb}
\usepackage{ascmac}
\usepackage{amsthm}

\makeatletter

\@addtoreset{equation}{section}
\makeatother 

\theoremstyle{plain}
\newtheorem{thm}{Theorem}[section]
\newtheorem*{thm*}{Theorem}
\newtheorem{prop}[thm]{Proposition}
\newtheorem{lem}[thm]{Lemma}
\newtheorem{cor}[thm]{Corollary}

\theoremstyle{definition}
\theoremstyle{remark}
\newtheorem{rem}[thm]{Remark}%[section]

\renewcommand{\epsilon}{\varepsilon}

\newcommand{\Rm}{\operatorname{Rm}}

\newcommand{\Ric}{\operatorname{Ric}}

\newcommand{\Div}{\operatorname{div}}

\newcommand{\tr}{\operatorname{tr}}

\usepackage{latexsym}

\makeatletter
\@namedef{subjclassname@2020}{%
  \textup{2020} Mathematics Subject Classification}
\makeatother

\title[Gaussian heat kernel estimates along super Ricci flow]{Gaussian heat kernel estimates of Bamler-Zhang type along super Ricci flow}

\author{Keita Kunikawa}
\address{Department of Mathematical Science, Faculty of Science and Technology, Tokushima University, 2-1 Minamijosanjima, Tokushima, 770-8506, Japan}
\email{kunikawa@tokushima-u.ac.jp}

\author{Yohei Sakurai}
\address{Department of Mathematics, Saitama University, 255 Shimo-Okubo, Sakura-ku, Saitama-City, Saitama, 338-8570, Japan}
\email{ysakurai@rimath.saitama-u.ac.jp}

\subjclass[2020]{Primary 53E20; Secondly 58J35}
\keywords{Super Ricci flow; M\"uller quantity; Distance distortion estimate; Cutoff function; Mean value inequality; Gaussian heat kernel estimate}
\date{December 30, 2024}

\begin{document}
\maketitle

\begin{abstract}
Bamler-Zhang have developed geometric analysis on Ricci flow with scalar curvature bound.
The aim of this paper is to extend their work to various geometric flows.
We generalize some of their results to super Ricci flow whose M\"uller quantity is non-negative,
and obtain Gaussian heat kernel estimates.
\end{abstract}

%%%%%%%%%%%%%%%%%%%%%%%%%%%%
%%%%%%%%%%%%%%%%%%%%%%%%%%%%
%%%%%%%%%%%%%%%%%%%%%%%%%%%%
\section{Introduction}

%%%%%%%%%%%%%%%%%%%%%%%%%%%%
\subsection{Backgrounds}\label{sec:Background}

Let $(M,g(t))_{t\in [0,T)}$ be a solution to the Ricci flow
\begin{equation}\label{eq:RF}
\partial_t g=-2\Ric
\end{equation}
on an $n$-dimensional compact smooth manifold with $T<+\infty$.
One of the common problems in the Ricci flow theory is to understand the blow-up behavior of curvature when $t \nearrow T$.
In a pioneering work by Hamilton \cite{H1},
he has proved that the supremum of the norm of the Riemannian curvature tensor blows up at $T$;
namely,
\begin{equation}\label{eq:blowup}
\lim_{t \nearrow T} \sup_{M}|\Rm|=+\infty.
\end{equation}
After that
under the same setting,
\v{S}e\v{s}um \cite{Se} has generalized \eqref{eq:blowup} to the Ricci curvature;
namely, she has proved
\begin{equation*}
\limsup_{t \nearrow T} \sup_{M}|\Ric|=+\infty.
\end{equation*}
It is natural to ask whether the same phenomenon occurs for the scalar curvature;
namely,
\begin{equation*}
\limsup_{t \nearrow T} \sup_{M}R=+\infty.
\end{equation*}
This is known to occur in the case of $n=2,3$ or in the K\"{a}hler case (see \cite{H2}, \cite{I}, \cite{Zh}),
but it seems to be still open in the higher dimensional and non-K\"{a}hler case.
From this point of view,
it is important to investigate the structure of Ricci flow with scalar curvature bound.

Based on the background,
geometric analysis on Ricci flow with scalar curvature bound has been developed by Cao-Tran \cite{CT}, Chen-Wang \cite{CW1}, \cite{CW2}, \cite{CW3}, \cite{CW4}, Wang \cite{W1}, \cite{W2}, Zhang \cite{Z2} and so on.
Bamler-Zhang \cite{BZ1} have made remarkable progress in this direction establishing useful tools such as distance distortion estimates,
a nice cutoff function,
and mean value inequalities.
By using them,
they have also concluded Gaussian heat kernel estimates,
backward pseudolocality theorem,
and strong $\varepsilon$-regularity theorem.
Recently,
their estimates have yielded numerous results in the sequel \cite{BZ2},
and a series of works by Bamler \cite{B0}, \cite{B1}, \cite{B2}, \cite{B3} while being improved.

%%%%%%%%%%%%%%%%%%%%%%%%%%%%
\subsection{Main results}\label{sec:Main results}

The purpose of this paper is to extend the results by Bamler-Zhang \cite{BZ1} to more general geometric flows.
We focus on a super solution to the Ricci flow \eqref{eq:RF}.
An $n$-dimensional compact manifold $(M,g(t))_{t\in [0,T)}$ equipped with a time-dependent metric is called \textit{super Ricci flow} when
\begin{equation*}\label{eq:SRF}
\partial_t g \geq -2\Ric.
\end{equation*}
McCann-Topping \cite{MT} have introduced this notion,
and studied the relation between the Ricci flow theory and optimal transport theory.
Recently,
this notion has begun to be investigated not only from the viewpoint of the Ricci flow theory but also from that of metric measure geometry (see e.g., \cite{K}, \cite{KS}, \cite{KS1}, \cite{S}).  

We examine super Ricci flow under the assumption that
its \textit{M\"uller quantity} is non-negative.
For a (time-dependent) vector field $V$,
the \textit{M\"uller quantity} is defined by the following formula (see \cite[Definition 1.3]{M1}):
\begin{equation*}
\mathcal{D}(V):=\partial_{t}H-\Delta H-2| h |^2+4\Div h(V)-2 \langle \nabla H,V \rangle+2\Ric(V,V)-2h(V,V),
\end{equation*}
where
\begin{equation*}
h:=-\frac{1}{2} \partial_t g,\quad H:=\tr h.
\end{equation*}
It is well-known that
several results for Ricci flow can be extended to more general geometric flow under the assumption that its M\"uller quantity is non-negative;
for instance, the monotonicity of $\mathcal{W}$-functional and the reduced geometry (more precisely, see Subsections \ref{sec:Entropy functional}, \ref{sec:Reduced geometry}).
Also,
there are some examples of super Ricci flow whose M\"uller quantity is non-negative (see \cite[Section 2]{M1}, \cite[Section 7]{FZ}):
\begin{enumerate}\setlength{\itemsep}{+1.0mm}
\item Ricci flow;
\item Ricci flow coupled with the heat equation, called \textit{List flow} (\cite{L});
\item Ricci flow coupled with the harmonic map heat flow, called \textit{M\"uller flow} (\cite{M2});
\item mean curvature flow for spacelike hypersurfaces in Lorentzian manifold of non-negative sectional curvature;
\item (scaled) twisted K\"{a}hler-Ricci flow.
\end{enumerate}

For $x,y \in M$ and $s,t\in [0,T)$ with $s< t$,
we denote by $G(x,t;y,s)$ the heat kernel for the heat equation (more precisely, see Subsection \ref{sec:Entropy functional}).
One of our main results is the following Gaussian heat kernel estimate,
which has been formulated by Bamler-Zhang \cite{BZ1} for Ricci flow (see \cite[Theorem 1.4]{BZ1}):
\begin{thm}\label{thm:mainhk} 
Let $(M,g(t))_{t\in [0,T)}$ be an $n$-dimensional compact super Ricci flow with $T<+\infty$ satisfying $\mathcal{D}(V)\geq 0$ for all vector fields $V$.
Then for any $A>0$,
there exist positive constants $\mathcal{C}_1,\mathcal{C}_2,\mathcal{C}_3,\mathcal{C}_4,\mathcal{C}_5,\mathcal{C}_6>0$ depending only on $n,T,g(0)$ and $A$ such that the following holds:
We assume $H \leq H_1$ for $H_1> 0$.
For $s,t \in [0,T)$ with $s<t$,
we suppose $t-s \leq A H_1^{-1}$ and $s \geq (t-s)/A$.
Then we have
\begin{align}\label{thm:mainhklow} 
G(x,t; y, s) &\geq \frac{\mathcal{C}_1}{(t-s)^{n/2}} \exp \left( {- \frac{\mathcal{C}_2\, d^2_s (x, y)}{t-s}} \right), \\ \label{thm:mainhkup} 
G(x, t; y, s) &\leq \frac{\mathcal{C}_3}{(t-s)^{n/2}} \exp \left({- \frac{\mathcal{C}_4 \,d^2_s (x, y)}{t-s}} \right),\\ \label{thm:maingrad} 
|\nabla_x G|(x,t;y,s)&\leq \frac{\mathcal{C}_5}{(t-s)^{(n+1)/2}} \exp \left({- \frac{\mathcal{C}_6 \,d^2_s (x, y)}{t-s}} \right).
\end{align} 
\end{thm}

\begin{rem}\label{rem:constant}
We give a remark for the dependence of the constants on $g(0)$.
More precisely,
they depend on the volume and Sobolev constant of $M$ with respect to $g(0)$,
and the minimum of $H$ at $t=0$,
which is derived from the logarithmic Sobolev inequality of Fang-Zheng \cite{FZ} (cf. \cite[Theorem 1.1]{FZ}, Subsection \ref{sec:Entropy functional}).
This remark can be applied throughout this paper.
\end{rem}

%%%%%%%%%%%%%%%%%%%%%%%%%%%%
%%%%%%%%%%%%%%%%%%%%%%%%%%%%
%%%%%%%%%%%%%%%%%%%%%%%%%%%%
\section{Preliminaries}\label{sec:Preliminaries}

We recall some previous works for general geometric flow.
Let $(M,g(t))_{t\in [0,T)}$ denote an $n$-dimensional compact manifold with a time-dependent metric with $T<+\infty$.

%%%%%%%%%%%%%%%%%%%%%%%%%%%%
\subsection{Notations and basics}\label{sec:Basics}

In the present subsection,
we introduce some notations,
and elementary facts that will be used frequently throughout this article.

We denote by $d_t$ and $m_t$ the Riemannian distance and Riemannian volume measure with respect to $g(t)$,
respectively.
When $t$ is clear from the context,
we omit it,
and write them as $d$ and $m$.
Notice that
the time derivative of $d m$ is equal to $-H\,dm$.
For $x_0\in M, r_0>0$ and $t_0\in [0,T)$,
we denote by $B(x_0,r_0,t_0)$ the open ball of radius $r_0$ centered at $x_0$ with respect to $g(t_0)$.
Let $Q^{+}(x_0, r_0,t_0)$ stand for the \textit{forward parabolic cube} defined as
\begin{equation*}\label{eq:cube}
 Q^+(x_0, r_0,t_0):=\{ (x, t) \in M \times [0,T) \mid x\in B(x_0,r_0,t),\, t\in [t_0,t_0+r^2_0] \}.
% Q^-(x_0, r_0,t_0):=\{ (x, t) \in M \times [0,T) \mid x\in B(x_0,r_0,t),\, t\in [t_0-r^2_0,t_0] \}.
\end{equation*}

We begin with the following elementary fact (see \cite[Lemma 3.2]{FZ}):
\begin{prop}[\cite{FZ}]\label{prop:lowerS}
Assume that
$(M,g(t))_{t\in [0,T)}$ satisfies $\mathcal{D}(V)\geq 0$ for all vector fields $V$.
Then we have
\begin{equation*}
H\geq -\frac{n}{2t}.
\end{equation*}
\end{prop}

\begin{rem}\label{rem:rescaling}
We also mention the rescaling argument.
For a positive $r_0>0$,
we will consider the parabolic rescaling
\begin{equation*}
\bar{g}(\tau):=r^{-2}_0\,g(r^2_0 \tau)=r^{-2}_0g(t)
\end{equation*}
with $\tau:=r^{-2}_0 t$.
We see that $\bar{h}=h,\bar{H}=r^2_0 H$ and $\overline{\mathcal{D}}(V)=r^4_0 \mathcal{D}(r^{-2}_0 V)$,
where $\bar{h}, \bar{H}, \overline{\mathcal{D}}(V)$ are corresponding objects for $\bar{g}(\tau)$.
In particular,
the super Ricci flow and the non-negativity of the M\"uller quantity are preserved under this rescaling.
\end{rem}

%%%%%%%%%%%%%%%%%%%%%%%%%%%%
\subsection{Monotonicity of entropy}\label{sec:Entropy functional}
The monotonicity of Perelman's $\mathcal{W}$-functional is one of the powerful tools in the Ricci flow theory (\cite{P}).
It is well-known that
such a monotonicity can be extended to general geometric flow under the non-negativity of the M\"uller quantity (see \cite[Theorem 3.1]{H}, \cite[Theorem 5.2]{GPT}, \cite[Lemma 3.1]{FZ}).
In this subsection,
we collect some results that can be derived from the monotonicity.

Fang-Zheng \cite{FZ} obtained a logarithmic Sobolev inequality by using the monotonicity of $\mathcal{W}$-functional,
and derived the following Sobolev inequality (see \cite[Theorems 1.1, 1.3]{FZ}, and cf. \cite{Y}, \cite{Z3}):
\begin{prop}[\cite{FZ}]\label{prop:Sob}
Assume that
$(M,g(t))_{t\in [0,T)}$ satisfies $\mathcal{D}(V)\geq 0$ for all vector fields $V$.
Then for all $t\in [0,T)$ and $u\in W^{1,2}(M)$ we have
\begin{equation*}
\left(\int_{M}\,|u|^{\frac{2n}{n-2}}\,dm\right)^{\frac{n-2}{n}}\leq C_1\int_{M}\,\left(|\nabla u|^2+\frac{H}{4}u^2 \right)\,dm+C_2 \,\int_{M}\,u^2\,dm,
\end{equation*}
where $C_1,C_2>0$ are positive constants depending only on $n,T$ and $g(0)$.
\end{prop}

They further concluded that
the following $\kappa$-noncollapsing estimate holds (see \cite[Theorem 6.1]{Y}, \cite[Theorem 5.1]{FZ}):
\begin{thm}[\cite{FZ}]\label{thm:noncollapse}
Assume that
$(M,g(t))_{t\in [0,T)}$ satisfies $\mathcal{D}(V)\geq 0$ for all vector fields $V$.
For $x\in M, r \in (0,\sqrt{T})$ and $t\in [0,T)$,
we assume $H \leq r^{-2}$ on $B(x,r,t)$.
Then we have
\begin{equation*}
m( B(x,r,t)) \geq \kappa\, r^n,
\end{equation*}
where $\kappa>0$ is a positive constant depending only on $n,T$ and $g(0)$.
\end{thm}

For $x,y \in M$ and $s,t\in [0,T)$ with $s< t$,
we denote by $G(x,t;y,s)$ the heat kernel for the heat equation;
namely,
for a fixed $(y,s) \in M \times [0,T)$,
it solves
\begin{equation*}
 (\partial_t - \Delta_x ) G(\cdot,\cdot;y,s) = 0,\quad \lim_{t \searrow s} G(\cdot,t;y,s)=\delta_y.
\end{equation*}
Notice that
$G(x,t;\cdot,\cdot)$ is the kernel for the conjugate heat equation;
namely,
for any $(x,t) \in M \times [0,T)$,
\begin{equation*}
(-\partial_s -\Delta_y +H)G(x,t;\cdot,\cdot) = 0,\quad \lim_{s \nearrow t} G(x,t;\cdot, s) = \delta_x.
\end{equation*}
They also showed the following upper bound of Zhang type (see \cite[(1.5)]{Z2}, \cite[Lemma 6.3]{FZ}):
\begin{prop}[\cite{FZ}]\label{prop:upper heat kernel}
Assume that
$(M,g(t))_{t\in [0,T)}$ satisfies $\mathcal{D}(V)\geq 0$ for all vector fields $V$.
Then for all $x,y\in M$ and $s,t \in [0,T)$ with $s< t$,
we have
\begin{equation*}
G(x,t;y,s) \leq \frac{C}{(t-s)^{n/2}},
\end{equation*}
where $C>0$ is a positive constant depending only on $n,T$ and $g(0)$.
\end{prop}
%\begin{proof}
%More precisely,
%they proved
%\begin{equation*}
%G(x,t;y,s)\leq \frac{\exp[2A_1T+A_2-T \min H^{-}_0+n]}{(4(t-s))^{\frac{n}{2}}},
%\end{equation*}
%where $H^{-}:=\min \{0,H\}\leq 0$.
%\end{proof}

%%%%%%%%%%%%%%%%%%%%%%%%%%%%
\subsection{Reduced geometry}\label{sec:Reduced geometry}
Perelman's reduced geometry is also a crucial concept in the Ricci flow theory (\cite{P}).
It is also well-known that
the reduced geometry can be extended to general geometric flow under the assumption that the M\"uller quantity is non-negative (see e.g., \cite{M1}, \cite{H}, \cite{Y1}, \cite{Y2}).
Here we present some results concerning the reduced geometry.

Let us recall the notion of the reduced distance.
Let $x,y\in M$ and $s,t\in [0,T)$ with $s<t$.
The \textit{$L$-distance} $L_{(x,t)}(y,s)$ from a space-time base point $(x,t)$ to $(y,s)$ is defined by
\begin{equation*}
L_{(x,t)}(y,s):=\inf_{\gamma}\int^{t}_{s}\sqrt{t-\xi}\left[ H +\left| \frac{d\gamma}{d\xi} \right|^{2} \right]\,d\xi,
\end{equation*}
where the infimum is taken over all curves $\gamma:[s,t]\to M$ with $\gamma(s)=y$ and $\gamma(t)=x$.
The \textit{reduced distance} $\ell_{(x,t)}(y,s)$, and its \textit{modified one} $\overline{L}_{(x,t)}(y,s)$ are defined as
\begin{equation*}
\ell_{(x,t)}(y,s):=\frac{1}{2\sqrt{t-s}}L_{(x,t)}(y,s),\quad \overline{L}_{(x,t)}(y,s):=4(t-s) \,\ell_{(x,t)}(y,s).
\end{equation*}
It is well-known that
the following fundamental inequality holds under the non-negativity of the M\"uller quantity (see \cite[Lemmas 5.2, 5.3]{M1}, and also \cite[Theorem 3.4]{H}, \cite[Subsection 2.3]{Y1}, \cite{Y2}):
\begin{prop}[\cite{M1}, \cite{H}, \cite{Y1}, \cite{Y2}]\label{prop:reduced heat estimate}
Assume that
$(M,g(t))_{t \in [0,T)}$ satisfies $\mathcal{D}(V)\geq 0$ for all vector fields $V$.
Then we have
\begin{equation}\label{eq:reduced heat estimate}
(-\partial_s+\Delta)  \overline{L}_{(x,t)} \leq 2n
\end{equation}
in the barrier sense.
\end{prop}

For the Ricci flow,
this estimate \eqref{eq:reduced heat estimate} together with the maximum principle leads us to the uniform upper bound of the minimum of the reduced distance (see \cite{P}, \cite[Lemmas 7.48, 7.50]{CCGG1}).
By the same argument,
Proposition \ref{prop:reduced heat estimate} gives us the following:
\begin{prop}\label{prop:reduced estimate}
Assume that
$(M,g(t))_{t \in [0,T)}$ satisfies $\mathcal{D}(V)\geq 0$ for all vector fields $V$.
Then for all $x\in M$ and $s,t\in [0,T)$ with $s<t$ we have
\begin{equation*}
\min_{y\in M}  \ell_{(x,t)}(y,s) \leq \frac{n}{2}.
\end{equation*}
\end{prop}

We also mention the following relation between the reduced distance and heat kernel (see \cite[Lemma 2.4]{CGT}):
\begin{prop}[\cite{CGT}]\label{prop:reduced kernel estimate}
Assume that
$(M,g(t))_{t \in [0,T)}$ satisfies $\mathcal{D}(V)\geq 0$ for all vector fields $V$.
Then for all $x,y\in M$ and $s,t\in [0,T)$ with $s<t$ we have
\begin{equation*}
G(x,t;y,s)\geq \frac{1}{(4\pi (t-s))^{n/2}} \exp\left(- \ell_{(x,t)}(y,s)\right).
\end{equation*}
\end{prop}

%%%%%%%%%%%%%%%%%%%%%%%%%%%%
\subsection{Zhang type gradient estimates}\label{sec:Gradient estimates}

For the super Ricci flow,
one can derive the following gradient estimate of Zhang type (see \cite[Theorem 3.1]{Z1}, \cite[Lemma 6.5]{FZ}):
\begin{thm}[\cite{FZ}]\label{thm:Zhang}
Let $(M,g(t))_{t\in [0,T)}$ be a super Ricci flow.
For $t_0 \in [0,T)$,
let $u \in C^\infty (M \times (t_0,T))$ be a positive solution to the heat equation.
For $t_1,t_2\in (t_0,T)$ with $t_1<t_2$,
we assume $u\leq A$ for $A> 0$ on $M\times [t_1,t_2]$.
Then for all $x\in M$ and $t\in (t_1,t_2]$ we have
\begin{equation*}
\frac{|\nabla u|}{u}(x,t)  \leq \frac{1}{\sqrt{t-t_1}}\sqrt{\log \frac{A}{u(x,t)}}.
\end{equation*}
\end{thm}

Due to Theorem \ref{thm:Zhang},
we have the following Harnack inequality (see \cite{Z2}, \cite[Lemma 6.5]{FZ}):
\begin{cor}[\cite{FZ}]\label{cor:lemZhang}
Let $(M,g(t))_{t\in [0,T)}$ be a super Ricci flow.
For $t_0 \in [0,T)$,
let $u \in C^\infty (M \times (t_0,T))$ be a positive solution to the heat equation.
For $t_1,t_2\in (t_0,T)$ with $t_1<t_2$,
we assume $u\leq A$ for $A> 0$ on $M\times [t_1,t_2]$.
Then for all $x,y\in M$ we have
\begin{equation*}
u(y,t_2)\leq A^{1/2}\,u(x,t_2)^{1/2}\,\exp \left(\frac{d^2_{t_2}(x,y)}{4(t_2-t_1)}\right).
\end{equation*}
\end{cor}
%\begin{proof}
%In view of Theorem \ref{thm:Zhang},
%for every $x\in M$ we have
%\begin{equation*}
%\left|\nabla \sqrt{\log\frac{A}{u}}\right|(x,t_2)=\frac{1}{2}\frac{1}{ \sqrt{ \log\cfrac{A}{u(x,t_2)}} } \frac{|\nabla u|}{u}(x,t_2)\leq \frac{1}{2\sqrt{t_2-t_1}}.
%\end{equation*}
%It follows that
%for all $x,y\in M$ we have
%\begin{equation*}
%\sqrt{\log\frac{A}{u(x,t_2)}}\leq \sqrt{\log\frac{A}{u(y,t_2)}}+\frac{d_{t_2}(x,y)}{2\sqrt{t_2-t_1}}.
%\end{equation*}
%We conclude
%\begin{equation*}
%\log\frac{A}{u(x,t_2)} \leq \log\frac{A}{u(y,t_2)}+\frac{d_{t_2}(x,y)^2}{4(t_2-t_1)}+\sqrt{\log\frac{A}{u(y,t_2)}}\frac{d_{t_2}(x,y)}{\sqrt{t_2-t_1}}\leq 2\log\frac{A}{u(y,t_2)}+\frac{d_t(x,y)^2}{2(t_2-t_1)}.
%\end{equation*}
%We obtain the desired assertion.
%\end{proof}

Using Propositions \ref{prop:upper heat kernel}, \ref{prop:reduced kernel estimate} and Corollary \ref{cor:lemZhang},
Fang-Zheng \cite{FZ} have obtained the following lower Gaussian heat kernel estimate (see \cite{Z2}, \cite[Lemma 6.6]{FZ}):
\begin{thm}[\cite{FZ}]\label{thm:lowerGauss}
Let $(M,g(t))_{t\in [0,T)}$ be a super Ricci flow satisfying $\mathcal{D}(V)\geq 0$ for all vector fields $V$.
For $s,t \in [0,T)$ with $s< t$,
and for $H_1> 0$,
we assume $H\leq H_1$ on $M\times [s,t]$.
Then for all $x,y\in M$ we have
\begin{equation*}
G(x,t; y,s) \geq \frac{C}{(t-s)^{n/2}} \exp \bigg( {- \frac{ d^2_t(x,y)}{t-s}} \bigg),
\end{equation*}
where $C>0$ is a positive constant depending only on $n,T,g(0)$ and $(t-s)H_1$.
\end{thm}
%\begin{proof}
%Applying Corollary \ref{cor:lemZhang} to
%\begin{equation*}
%t_0=s,\quad  t_1=\frac{s+t}{2},\quad t_2=t,\quad u=G(\cdot,\cdot;y,s),\quad A=\sup_{M\times \left[\frac{s+t}{2},t\right]}u,
%\end{equation*}
%we see
%\begin{equation*}
%G(y,t;y,s)\leq A^{1/2}\,G(x,t;y,s)^{1/2}\,\exp\left(\frac{d_t(x,y)^2}{2(t-s)}\right),
%\end{equation*}
%and hence
%\begin{equation}\label{eq:lowerGauss1}
%G(x,t;y,s)\geq A^{-1}\,G(y,t;y,s)^2\,\exp\left(-\frac{d_t(x,y)^2}{t-s}\right).
%\end{equation}
%
%Proposition \ref{prop:upper heat kernel} tells us that
%for all $\xi \in [(s+t)/2,t]$ we have
%\begin{equation*}
%G(x,\xi;y,s)\leq \frac{C_1}{(\xi-s)^{n/2}}\leq \frac{C_2}{(t-s)^{n/2}}.
%\end{equation*}
%In particular,
%\begin{equation}\label{eq:lowerGauss2}
%A\leq \frac{C_2}{(t-s)^{n/2}}.
%\end{equation}
%On the other hand,
%Proposition \ref{prop:reduced kernel estimate} implies
%\begin{align}\label{eq:lowerGauss3}
%G(y,t;y,s)&\geq \frac{1}{(4\pi (t-s))^{n/2}} \exp\left( - \ell_{(y,s)}(y,t)  \right)\\ \notag
%&\geq \frac{1}{(4\pi (t-s))^{n/2}} \exp \left({-\frac{1}{2\sqrt{t-s}}\int^{t}_{s}\,\sqrt{t-\xi}H(y,\xi)\,d\xi}\right)\\ \notag
%&\geq \frac{1}{(4\pi (t-s))^{n/2}} \exp \left(-{\frac{(t-s)H_1}{3}}\right).
%\end{align}
%Combining \eqref{eq:lowerGauss1}, \eqref{eq:lowerGauss2}, \eqref{eq:lowerGauss3},
%we arrive at the desired estimate.
%\end{proof}

From Propositions \ref{prop:upper heat kernel}, \ref{prop:reduced kernel estimate} and Corollary \ref{cor:lemZhang},
they also concluded the following $\kappa$-noninflating theorem of Zhang type (see \cite[Theorem 1.1]{Z2}, \cite[Theorem 6.1]{FZ}):
\begin{thm}[\cite{FZ}]\label{thm:inflat}
Let $(M,g(t))_{t\in [0,T)}$ be a super Ricci flow satisfying $\mathcal{D}(V)\geq 0$ for all vector fields $V$.
For $x_0\in M, t_0\in (0,T),r_0\in (0,\sqrt{t_0})$ and $A>0$,
we assume $H\leq A/(t_0-t)$ on $B(x_0,r_0,t_0)\times [t_0-r^2_0,t_0]$.
Then we have
\begin{equation*}
 m( B(x_0,r_0,t_0)) \leq \kappa\, r^n_0,
\end{equation*}
where $\kappa>0$ is a positive constant depending only on $n,T,g(0)$ and $A$.
\end{thm}
\subsection{Characterization of super Ricci flow}\label{sec:Average bound}
 It is well-known that
 super Ricci flow can be characterized by a gradient estimate,
 logarithmic Sobolev inequality,
 and Poincar\'e inequality for the heat semigroup,
 which can be regarded as a time-dependent version of the equivalence of a lower Ricci curvature bound established by Bakry-\'Emery \cite{BE}, Bakry-Ledoux \cite{BL} (see e.g., \cite{HN}, \cite{LL}, \cite{HaN}, \cite{KS1},
 and also \cite{MT}, \cite{S} for other characterizations).
 In this subsection,
 we present some results that can be deduced from such a characterization.

Let $x,y\in M$,
and let $s,t\in [0,T)$ with $s< t$.
The \textit{heat kernel measure} $\nu_{(x,t)}(y,s)$ is defined by
\begin{equation*}
d\nu_{(x,t)}(y,s):=G(x,t;y,s)\,dm_{s}(y).
\end{equation*}
We possess the following logarithmic Sobolev inequality (see \cite[Theorem 1.1]{LL}, \cite[Theorem 1.5]{HaN}, \cite[Theorem 1.3]{KS1}, and cf. \cite[Theorem 1.10 (2)]{HN}):
\begin{thm}[\cite{LL}, \cite{HaN}, \cite{KS1}]\label{thm:BakryLS}
Let $(M,g(t))_{t\in [0,T)}$ be a super Ricci flow.
Then for all $u \in C^{\infty}(M)$ with $\int_{M}\,u^2(y)\,d\nu_{(x,t)}(y,s)=1$ we have
\begin{equation*}
\int_{M}\,u^2 \log u^2(y)\,d\nu_{(x,t)}(y,s)\leq 4(t-s) \int_{M}\,|\nabla u|^2(y)\,d\nu_{(x,t)}(y,s).
\end{equation*}
\end{thm}

Theorem \ref{thm:BakryLS} together with the method of the proof of \cite[Theorem 1.13]{HN} implies the following Gaussian concentration inequality (cf. \cite[Theorem 1.13]{HN}):
\begin{cor}\label{cor:GaussConc}
Let $(M,g(t))_{t\in [0,T)}$ be a super Ricci flow.
Let $\nu:=\nu_{(x,t)}(\cdot,s)$.
Then for all $\Omega_1,\Omega_2\subset M$ we have
\begin{equation*}
\nu(\Omega_1)\nu(\Omega_2)\leq \exp\left(-\frac{d^2_s(\Omega_1,\Omega_2)}{8(t-s)}   \right).
\end{equation*}
\end{cor}

Corollary \ref{cor:GaussConc} together with the method of the proof of \cite[(1.20)]{HN} tells us the following average bound for the heat kernel (cf. \cite[(1.20)]{HN}):
\begin{prop}\label{prop:Average}
Let $(M,g(t))_{t\in [0,T)}$ be a super Ricci flow.
Let $x,y\in M$,
and let $s,t\in [0,T)$ with $s< t$.
For $c>0$ we assume $m(B(y, c\sqrt{t-s},s))\geq \kappa (t-s)^{n/2}$.
Then
\begin{align*}
&\quad \,\,\frac{1}{ m(B(y, c\sqrt{t-s},s))} \int_{B(y,c\sqrt{t-s},s)} G(x, t; z, s) \,dm_s(z) \\
&\leq \bigg( \int_{B(x,c\sqrt{t-s},s)} G(x,t; z, s ) \,dm_s (z) \bigg)^{-1} \frac{\kappa^{-1} \,e^{c^2/2}}{(t-s)^{n/2}} \exp \left( - \frac{d^2_{s}(x,y)}{16(t-s)}  \right).
\end{align*}
\end{prop}
%\begin{proof}
%Applying Corollary \ref{cor:GaussConc} to $\Omega_1=B(y,c\sqrt{t-s},s)$ and $\Omega_2=B(x,c\sqrt{t-s},s)$,
%we obtain
%\begin{equation*}
%\int_{B(y,c\sqrt{t-s},s)} G(x, t; z, s) dm_s(z)\,\int_{B(x,c\sqrt{t-s},s)} G(x, t; z, s) dm_s(z)\leq \exp\left(-\frac{d_s(\Omega_1,\Omega_2)^2}{8(t-s)}   \right).
%\end{equation*}
%Since $d_s(\Omega_1,\Omega_2)\geq d_s(x,y)-2c\sqrt{t-s}$,
%we see
%\begin{equation*}
%-d_s(\Omega_1,\Omega_2)^2\leq -\frac{d_s(x,y)^2}{2}+4c^2(t-s),
%\end{equation*}
%and hence
%\begin{equation*}
%\int_{B(y,c\sqrt{t-s},s)} G(x, t; z, s) dm_s(z)\,\int_{B(x,c\sqrt{t-s},s)} G(x, t; z, s) dm_s(z)\leq e^{c^2/2}\exp\left(-\frac{d_s(x,y)^2}{16(t-s)}   \right).
%\end{equation*}
%Now,
%the noncollapsing property implies the desired conclusion.
%\end{proof}

In \cite{HN},
a similar inequality has been stated only for $c=1$.
For the sake of our argument,
we present a slightly general form.
The method is completely same as that of \cite[(1.20)]{HN}.

%%%%%%%%%%%%%%%%%%%%%%%%%%%%
%%%%%%%%%%%%%%%%%%%%%%%%%%%%
%%%%%%%%%%%%%%%%%%%%%%%%%%%%
\section{Distance distortion estimate}\label{sec:distance}

In what follows,
we denote by $(M,g(t))_{t\in [0,T)}$ an $n$-dimensional compact super Ricci flow with $T<+\infty$ satisfying $\mathcal{D}(V)\geq 0$ for all vector fields $V$.
Unless otherwise stated,
positive constants which appear in the proof will depend only on $n,T$ and $g(0)$ (cf. Remark \ref{rem:constant}).

The aim of this section is to prove a distance distortion estimate.
We start with the following key lemma,
which has been obtained by Bamler-Zhang \cite{BZ1} for Ricci flow (see \cite[Lemma 3.1]{BZ1}):
\begin{lem}\label{lem:distance basic}
There exists a positive constant $C_n>0$ depending only on $n$ such that the following holds:
For $t_0 \in [0,T)$,
let $u \in C^\infty (M \times (t_0,T))$ be a positive solution to the heat equation.
For $t_1,t_2\in (t_0,T)$ with $t_1<t_2$,
we assume $u\leq A$ for $A> 0$ on $M\times [t_1,t_2]$.
Then for all $x\in M$ and $t\in (t_1,t_2]$ we have
\begin{equation*}
\left( |\Delta u | + \frac{| \nabla u |^2}{u} - A H \right) (x, t) \leq \frac{ C_n A}{t-t_1}.
\end{equation*}
\end{lem}
\begin{proof}
We may assume $A=1$.
We first estimate
\begin{equation*}
L_1:= - \Delta u + \frac{|\nabla u|^2}{u}-H.
\end{equation*}
Using the Bochner formula,
we deduce the following (cf. \cite[Proposition 5.1.1]{Z4}):
\begin{align}\label{eq:evolution}
 (\partial_t - \Delta ) \Delta u &=\partial_t \Delta u-\Delta \partial_t u  = 2\langle \nabla^2u, h \rangle+2\Div h(\nabla u)-\langle \nabla H,\nabla u \rangle, \\ \notag
 (\partial_t - \Delta ) \frac{|\nabla u|^2}{u} &= - \frac{2}{u} \left( \left|\nabla^2 u -\frac{du\otimes du}{u}\right|^2+\Ric(\nabla u,\nabla u)-h(\nabla u,\nabla u) \right),\\ \notag
 (\partial_t - \Delta ) H &= 2 | h |^2+\mathcal{D}(0).
\end{align}
Therefore, by $u\leq 1, \Ric\geq h$ and $\mathcal{D}(V)\geq 0$ we obtain
\begin{align*}
(\partial_t - \Delta)L_1&=- \frac{2}{u} \left| \nabla^2 u -\frac{ du \otimes du}{u} \right|^2-2\langle \nabla^2 u,h\rangle-2 | h |^2\\
                               &\quad- 2\Div h(\nabla u)+\langle \nabla H,\nabla u \rangle-\frac{2}{u}\left(\Ric(\nabla u,\nabla u)-h(\nabla u,\nabla u) \right)-\mathcal{D}(0)\\
                                 &\leq -  \left| \nabla^2 u -\frac{ du \otimes du}{u} \right|^2-2\langle \nabla^2 u,h\rangle-2 | h |^2\\
                                 &\quad- 2\Div h(\nabla u)+\langle \nabla H,\nabla u \rangle-(\Ric(\nabla u,\nabla u)-h(\nabla u,\nabla u))-\mathcal{D}(0)\\
                                  &= - \left| \nabla^2 u -\frac{ du \otimes du}{u}+h \right|^2- \left| \frac{du\otimes du}{u}+h \right|^2+\frac{|\nabla u|^4}{u^2}-\frac{1}{2}(\mathcal{D}(\nabla u)+\mathcal{D}(0))\\
                                   &\leq  - \frac{1}{n} \left( \Delta u -\frac{ |\nabla u|^2}{u}+H \right)^2+\frac{|\nabla u|^4}{u^2}=  - \frac{1}{n} L^2_1+\frac{|\nabla u|^4}{u^2}.
\end{align*}
Theorem \ref{thm:Zhang} tells us that
\begin{equation}\label{eq:distance basic1}
\frac{|\nabla u|^2}{u}  \leq \frac{u}{t-t_1}\log \frac{1}{u}\leq \frac{1}{e(t-t_1)}
\end{equation}
on $M\times (t_1,t_2]$,
and hence
\begin{equation*}
(\partial_t - \Delta)L_1\leq - \frac{1}{n} L^2_1+\frac{1}{e^2(t-t_1)^2}.
\end{equation*}
Let $C_{n,1}>0$ be a positive constant determined by
\begin{equation*}
\frac{C_{n,1}+e^{-2}}{C_{n,1}^2} = \frac{1}{n}
\end{equation*}
depending only on $n$.
Note that
\begin{equation*}
(\partial_t - \Delta ) \left(\frac{C_{n,1}}{t-t_1}\right)=- \frac{1}{n} \left(\frac{C_{n,1}}{t-t_1}\right)^2 + \frac{1}{e^2 (t-t_1)^2}.
\end{equation*}
It follows that
\begin{equation*}
 (\partial_t - \Delta ) \left( L_1 - \frac{C_{n,1}}{t-t_1} \right) \leq - \frac{1}{n} \left(L_1+ \frac{C_{n,1}}{t-t_1} \right) \left(L_1- \frac{C_{n,1}}{t-t_1} \right).
 \end{equation*}
The maximum principle leads us to $L_1 \leq C_{n,1}/(t-t_1)$.

We next estimate
\begin{equation*}
L_2:= \Delta u + \frac{|\nabla u|^2}{u}-H.
\end{equation*}
Using \eqref{eq:evolution}, $u\leq 1, \Ric\geq h$ and $\mathcal{D}(V)\geq 0$ again, one can calculate
\begin{align*}
(\partial_t - \Delta)L_2&=- \frac{2}{u} \left| \nabla^2 u -\frac{ du \otimes du}{u} \right|^2+2\langle \nabla^2 u,h\rangle-2 | h |^2\\
                               &\quad +2\Div h(\nabla u)-\langle \nabla H,\nabla u \rangle-\frac{2}{u}\left(\Ric(\nabla u,\nabla u)-h(\nabla u,\nabla u) \right)-\mathcal{D}(0)\\
                               &\leq - \left| \nabla^2 u -\frac{ du \otimes du}{u} \right|^2+2\langle \nabla^2 u,h\rangle-2 | h |^2\\
                               &\quad -2\Div h(-\nabla u)+\langle \nabla H,(-\nabla u) \rangle-\left(\Ric(\nabla u,\nabla u)-h(\nabla u,\nabla u) \right)-\mathcal{D}(0)\\
                              &=- \left| \nabla^2 u -\frac{ du \otimes du}{u} -h\right|^2-\left| \frac{du \otimes du}{u}-h \right|^2+\frac{|\nabla u|^4}{u^2}-\frac{1}{2}\left(\mathcal{D}(-\nabla u)+\mathcal{D}(0) \right)\\
                             &\leq -\frac{1}{n} \left( \Delta u-\frac{|\nabla u|^2}{u}-H  \right)^2+\frac{|\nabla u|^4}{u^2}\\
                             &=-\frac{1}{2n}\left( L_2-4\frac{|\nabla u|^2}{u}   \right)^2-\frac{1}{2n}L^2_2+\left( 1+\frac{4}{n}  \right)\frac{|\nabla u|^4}{u^2}
                             \leq -\frac{1}{2n}L^2_2+\left( 1+\frac{4}{n}  \right)\frac{|\nabla u|^4}{u^2}.
\end{align*}
The estimate \eqref{eq:distance basic1} implies
\begin{equation*}
(\partial_t - \Delta)L_2\leq - \frac{1}{n} L^2_2+\left( 1+\frac{4}{n} \right)\frac{1}{e^2(t-t_1)^2}
\end{equation*}
over $M\times (t_1,t_2]$.
Let $C_{n,2}>0$ be a constant determined by
\begin{equation*}
C^{-1}_{n,2}+\left( 1+\frac{4}{n} \right)\frac{1}{e^2}C^{-2}_{n,2} = \frac{1}{2n},
\end{equation*}
which depends only on $n$.
We notice that
\begin{equation*}
(\partial_t - \Delta ) \left(\frac{C_{n,2}}{t-t_1}\right)=- \frac{1}{2n} \left(\frac{C_{n,2}}{t-t_1}\right)^2 + \left(  1+\frac{4}{n} \right)\frac{1}{e^2 (t-t_1)^2};
\end{equation*}
in particular,
\begin{equation*}
 (\partial_t - \Delta ) \left( L_2 - \frac{C_{n,2}}{t-t_1} \right) \leq - \frac{1}{2n} \left(L_2+ \frac{C_{n,2}}{t-t_1} \right) \left(L_2- \frac{C_{n,2}}{t-t_1} \right).
 \end{equation*}
Due to the maximum principle,
we obtain $L_2 \leq C_{n,2}/(t-t_1)$.
This completes the proof.
\end{proof}

We are now in a position to prove the following distance distortion estimate,
which has been established by Bamler-Zhang \cite{BZ1} for Ricci flow (see \cite[Theorem 1.1]{BZ1}):
\begin{thm}\label{thm:distcon}
There exists a positive constant $\alpha \in (0,1)$ depending only on $n,T$ and $g(0)$ such that the following holds:
For $t_0 \in (0,T)$ and $r_0\in (0,\sqrt{t_0}]$,
we assume $H \leq r_0^{-2}$.
Let $x_0, y_0 \in M$ satisfy $d_{t_0} (x_0, y_0 ) \geq r_0$.
Then for all $t\in [0,T)$ with $|t-t_0|\leq \alpha r^2_0$,
\begin{equation*}
\alpha d_{t_0} (x_0, y_0) \leq d_t (x_0, y_0) \leq \alpha^{-1} d_{t_0} (x_0, y_0).
\end{equation*}
\end{thm}
\begin{proof}
By the argument of \cite[Theorem 1.1]{BZ1},
it suffices to prove the upper bound for $d_t(x_0,y_0)$ only in the case of $d_{t_0}(x_0, y_0) \leq 2 r_0$.
By the parabolic rescaling,
it is enough to show the claim when $r_0 = 1$ (see Remark \ref{rem:rescaling}).
Then we have $t_0 \geq 1$ and $H \leq 1$.
We further possess $H \geq -n$ on $M \times [t_0 -1/2,t_0+1/2]$ by virtue of Proposition \ref{prop:lowerS}.
Let $\gamma:[0,1] \to M$ be a minimal geodesic from $x_0$ to $y_0$ with respect to $g(t_0)$.
By Proposition \ref{prop:reduced estimate},
\begin{equation}\label{eq:reduced10}
\ell_{(x_0, t_0)} \left(z, t_0 - \frac{1}{2} \right) \leq  \frac{n}{2}
\end{equation}
for some $z \in M$.
We put $G(x,t):=G(x,t; z, t_0 - 1/2)$.

Due to Proposition \ref{prop:upper heat kernel},
on $M \times [t_0 - 1/4, t_0+1/4]$ it holds that
\begin{equation}\label{eq:distcon1}
 G(x, t) \leq \frac{C_1}{(t-(t_0 - 1/2))^{n/2}}\leq C_2.
\end{equation}
Thanks to Proposition \ref{prop:reduced kernel estimate} and \eqref{eq:reduced10},
\begin{equation}\label{eq:distcon2}
 G(x_0, t_0) \geq \frac{1}{(4\pi (t_0 - (t_0 -1/2))^{n/2}} \exp\left(- \ell_{(x_0, t_0)} (z, t_0 - 1/2)\right) \geq C_{3}.
\end{equation}
By \eqref{eq:distcon1}, Corollary \ref{cor:lemZhang}, \eqref{eq:distcon2} and the assumption $d_{t_0}(x_0,y_0)\leq 2$,
for all $x\in \gamma([0,1])$,
\begin{align}\label{eq:distcon3}
 G(x,t_0) &\geq C^{-1}_2 G(x_0,t_0)^2\,\exp \left(-\frac{d^2_{t_0}(x_0,x)}{2(t_0-(t_0-1/4))}\right)\\ \notag
              &\geq C^{-1}_2\,C^2_{3}\,\exp \left(-2d^2_{t_0}(x_0,y_0)\right)\geq e^{-8}\,C^{-1}_2\,C^2_{3}=:C_4.
\end{align}
For the dimensional constant $C_{n}>0$ obtained in Lemma \ref{lem:distance basic},
we set
\begin{equation*}
\alpha_0:=\min  \left\{\frac{1}{8},\frac{C_4}{4C_2(1+8C_{n})}\right\}.
\end{equation*}
Using \eqref{eq:distcon3}, \eqref{eq:distcon1}, Lemma \ref{lem:distance basic}, $H\leq 1$,
for all $x \in \gamma([0,1])$ and $t \in [t_0 - \alpha_0, t_0+\alpha_0]$,
we see
\begin{align}\label{eq:distcon4}
G(x,t) &\geq  G(x,t_0)-\int^{t_0+\alpha_0}_{t_0-\alpha_0}\,\left|\partial_t G\right|(x,t)\,dt=G(x,t_0)-\int^{t_0+\alpha_0}_{t_0-\alpha_0}\,\left|\Delta G\right|(x,t)\,dt\\ \notag
         &\geq  C_4-C_2\int^{t_0+\alpha_0}_{t_0-\alpha_0}\,\left( H(x,t) +\frac{C_{n}}{t-(t_0-1/4)} \right)\,dt\\ \notag
         &\geq C_4-2\alpha_0 \,C_2 \left( 1+\frac{C_{n}}{1/4-\alpha_0} \right)\geq C_4-2\alpha_0 \,C_2 \left( 1+8C_{n}\right)\geq \frac{C_4}{2}.
\end{align}
By \eqref{eq:distcon1}, Corollary \ref{cor:lemZhang}, \eqref{eq:distcon4},
for all $s \in [0,1], t \in [t_0 - \alpha_0/2, t_0+\alpha_0/2]$ and $x\in B(\gamma(s),\sqrt{\alpha_0},t)$ we have
\begin{equation*}
G(x,t) \geq C^{-1}_2  G(\gamma(s),t)^2\,\exp \left(-\frac{d^2_{t}(x,\gamma(s))}{2(t-(t_0-\alpha_0))}\right)\geq \frac{C^{-1}_2 C^2_4}{4}e^{-1}=:C_5.
\end{equation*}

Let $\{x_i\}^{N}_{i=1}$ denote a maximal $(2\sqrt{\alpha_0})$-separated set in $\gamma([0,1])$ with respect to $g(t)$.
The argument in the proof of \cite[Theorem 1.1]{BZ1} tells us that $d_t (x_0, y_0) \leq 4N\sqrt{\alpha_0}$.
On $(t_0-1/2,t_0+1/2]$, we see
\begin{equation*}
\frac{d}{dt} \int_{M}\, G \,dm = \int_{M} \,\left( \Delta G - H\, G \right)\, dm \leq n \int_{M} \,G\, dm
\end{equation*}
since $H \geq -n$ on $M \times [t_0 - 1/2, t_0+1/2]$, and hence
\begin{equation}\label{eq:L1}
\int_{M}\, G\, dm \leq \exp \left(n (t - (t_0 - 1/2))\right)\leq e^{n}.
\end{equation}
Combining \eqref{eq:L1} and Theorem \ref{thm:noncollapse},
we see that
if $t \in [t_0 - \alpha_0/2, t_0+\alpha_0/2]$,
then
\begin{equation*}
e^{n} \geq \int_{M}\, G(\cdot, t)\, dm \geq \sum_{i=1}^N \int_{B(x_i,\sqrt{\alpha_0},t)}\, G(\cdot, t)\, dm \geq N( \kappa \alpha^{n/2}_0)C_5.
\end{equation*}
We conclude that for all $t \in [t_0 - \alpha_0/2, t_0+\alpha_0/2]$,
\begin{equation*}
d_t (x_0, y_0) \leq \frac{4e^{n}}{\kappa \alpha^{(n-1)/2}_0 C_5} < \frac{8e^{n}}{\kappa \alpha^{(n-1)/2}_0 C_5} \,d_{t_0}(x_0, y_0).
\end{equation*}
Setting
\begin{equation*}
\alpha:=\min \left\{\frac{\alpha_0}{2}, \frac{\kappa \alpha^{(n-1)/2}_0 C_5}{8e^{n}} \right\},
\end{equation*}
we arrive at the desired upper bound.
Thus we complete the proof.
\end{proof}

%%%%%%%%%%%%%%%%%%%%%%%%%%%%
%%%%%%%%%%%%%%%%%%%%%%%%%%%%
%%%%%%%%%%%%%%%%%%%%%%%%%%%%
\section{Cutoff function}\label{sec:cutoff}

Based on Theorem \ref{thm:distcon},
we construct the following cutoff function,
which has been obtained by Bamler-Zhang \cite{BZ1} for Ricci flow (see \cite[Theorem 1.3]{BZ1}):
\begin{thm}\label{thm:cutoff}
There exists a constant $\rho \in (0,1)$ depending only on $n,T$ and $g(0)$ such that the following holds:
Let $t_0\in (0,T)$ and $r_0 \in (0,\sqrt{t_0}]$.
We assume $H \leq r_0^{-2}$.
Then for any $x_0\in M$ and $\tau \in (0,  \rho^2 r_0^2]$,
there is $\phi \in C^\infty (M \times [t_0 - \tau,t_0])$ satisfying the following properties:
\begin{enumerate}\setlength{\itemsep}{+0.7mm}
\item $0 \leq \phi < 1$;
\item $\phi \geq \rho$ on $B(x_0,\rho r_0,t_0) \times [t_0-\tau, t_0]$;
\item $\phi = 0$ on $(M \setminus B(x_0, r_0,t_0) ) \times [t_0 - \tau, t_0]$;
\item $ \vert \nabla \phi \vert \leq r_0^{-1}$ and $\vert \partial_t \phi \vert + \vert \Delta \phi \vert \leq r_0^{-2}$.
\end{enumerate}
\end{thm}
\begin{proof}
By the rescaling,
we may assume $r_0 = 1$ (see Remark \ref{rem:rescaling}).
We possess $t_0 \geq 1$ and $H \leq 1$.
Moreover, $H \geq -n$ on $M \times [t_0 - 1/2, t_0]$ by Proposition \ref{prop:lowerS}.
Let $\theta \in (0,1/2), \tau\in (0,\theta/4)$, which will be determined later.

Due to Proposition \ref{prop:reduced estimate},
there exists $z \in M$ such that
\begin{equation*}
\ell_{(x_0, t_0)} (z, t_0 - \theta ) \leq  \frac{n}2.
\end{equation*}
Set $G(x,t):=G(x,t;z,t_0-\theta)$.
Thanks to Proposition \ref{prop:reduced kernel estimate},
\begin{align}\label{eq:cutoff11}
G(x_0, t_0) &\geq \frac{1}{(4\pi ( t_0 - (t_0 - \theta))^{n/2}} \exp \left(-\ell_{(x_0, t_0)} (z, t_0 - \theta)\right)\\ \notag
                 &\geq  \frac{1}{(4\pi  \theta)^{n/2}} e^{-n/2} =:C_{n} \theta^{-n/2},
\end{align}
where $C_n>0$ depends only on $n$. 
We define
\begin{equation*}
\Omega := \left\{ x \in M \mid G(x,t_0) >  C_{n} \theta^{-n/2}/2 \right\}. %open manifold
\end{equation*}
From \eqref{eq:cutoff11} we derive $x_0 \in \Omega$.
Let $\Omega_0 \subset \Omega$ stand for the connected component of $\Omega$ containing $x_0$. % connected open manifold; in particular, path connected
We prove that
for any sufficiently small $\theta$,
we have $\Omega_0 \subset B(x_0, 1,t_0)$ by contradiction.
Let us assume that
$\Omega_0$ is not contained in $B(x_0, 1,t_0)$.
Then we can take a curve $\gamma : [0,1] \to \Omega_0$ from $x_0$ to a point in $\partial B(x_0, 1,t_0) \cap \Omega_0$.
Let $\{x_i\}^{N}_{i=1}$ be a maximal $(2\sqrt{\theta})$-separated set in $\gamma([0,1])$ with respect to $g(t_0)$.
The argument in the proof of \cite[Theorem 1.3]{BZ1} tells us that
\begin{equation}\label{eq:lowerN}
N \geq \frac{1}{4 \sqrt{\theta}}.
\end{equation}
Proposition \ref{prop:upper heat kernel} tells us that
on $M\times [t_0 - \theta/2, t_0]$ we see
\begin{equation}\label{eq:cutoff1}
 G(x,t) \leq \frac{C_1}{(t-(t_0-\theta))^{n/2}}\leq  C_2 \theta^{-n/2}.
\end{equation}
Due to \eqref{eq:cutoff1}, Corollary \ref{cor:lemZhang},
for all $y \in \Omega_0$ and $z \in B(y, \sqrt{\theta},t_0)$ we have
\begin{equation}\label{eq:cutoff2}
G(z,t_0)\geq C^{-1}_2\,\theta^{n/2}\,G(y,t_0)^2\,\exp\left(-\frac{d_{t_0}(y,z)^2}{2\theta}\right)\geq  \frac{e^{-1/2}}{2}C^{-1}_2 C_{n} \theta^{-n/2}=:C_3 \theta^{-n/2}.
\end{equation}
On the other hand,
on $(t_0-\theta,t_0]$ we see
\begin{equation*}
\frac{d}{dt} \int_{M}\, G \,dm = \int_{M} \,\left( \Delta G - H\, G \right)\, dm \leq n \int_{M} \,G\, dm
\end{equation*}
since $H \geq -n$ on $M \times [t_0 - \theta, t_0]$, and hence
\begin{equation}\label{eq:cutoff3}
\int_{M}\, G\, dm \leq \exp \left(n (t - (t_0 - \theta))\right)\leq e^{n\theta}\leq e^{n/2}.
\end{equation}
Combining \eqref{eq:lowerN}, \eqref{eq:cutoff2}, \eqref{eq:cutoff3} and Theorem \ref{thm:noncollapse},
we obtain
\begin{equation*}
e^{n/2} \geq \int_{M} G(\cdot, t_0) dm \geq \sum_{i =1}^N \int_{B(x_i,\sqrt{\theta},t_0)} G(\cdot, t_0)\,dm \geq N (C_3 \theta^{-n/2}) ( \kappa \theta^{n/2}) = N \kappa C_3  \geq \frac{\kappa C_3}{4 \sqrt{\theta}}.
\end{equation*}
If $\theta < 16^{-1} e^{-n} \kappa^2 C^2_3$,
then this is a contradiction.
Hereafter,
we fix $\theta \in (0, 16^{-1} e^{-n} \kappa^2 C^2_3)$ such that $\Omega_0 \subset B(x_0, 1,t_0)$.

We now use \eqref{eq:cutoff1} and Lemma \ref{lem:distance basic}.
On $M\times [t_0-\tau,t_0]$ it holds that
\begin{align}\label{eq:cutoff4}
|\partial_t G|(x,t)&\leq  C_2 \theta^{-n/2}\left(H(x,t) +\frac{C_{4}}{t-(t_0-\theta/2)}\right)\leq C_2 \theta^{-n/2}(1+4C_{4} \theta^{-1})\\ \notag
&\leq  C_2(1+4C_{4}) \theta^{-n/2-1}=:C_5 \theta^{-n/2-1}.
\end{align}
Note that
$C_5>C_2$.
Let $\tau \in (0,0.1C_{n}C^{-1}_{5}\theta)$.
Then on $M\times [t_0-\tau,t_0]$,
\begin{equation*}
|G(x,t)-G(x,t_0)|\leq \int^{t_0}_{t_0-\tau}\,|\partial_t G|(x,t)\,dt \leq C_5 \theta^{-n/2-1} \tau<0.1 C_{n}\theta^{-n/2};
\end{equation*}
in particular, on $\partial \Omega_0\times [t_0-\tau,t_0]$ we possess
\begin{equation*}
G(x,t)< G(x,t_0)+0.1 C_{n}\theta^{-n/2}=0.5 C_{n} \theta^{-n/2}+0.1 C_{n}\theta^{-n/2}=0.6 C_{n}\theta^{-n/2}.
\end{equation*}
Now,
we can define a function $\hat{\psi}\in C^{0}(M\times [t_0-\tau,t_0])$ by
\begin{equation*}
\hat{\psi}(x,t):= \begin{cases}
                                  \max\{G(x,t)-0.6 C_{n}\theta^{-n/2} \} & \textrm{on}\ \Omega_0\times [t_0-\tau,t_0],\\
                                  0 & \textrm{on}\ (M\setminus \Omega_0) \times [t_0-\tau,t_0],
                                  \end{cases}
\end{equation*}
which is compactly supported on $\Omega_0\times [t_0-\tau,t_0]$.
This function enjoys the following properties:
\begin{enumerate}\setlength{\itemsep}{+0.7mm}
\item $0 \leq \hat{\psi} < C_5 \theta^{-n/2}$;
\item $\hat{\psi}(x_0,t_0)\geq 0.4 C_{n}\theta^{-n/2}$;
\item $ \vert \nabla \hat{\psi} \vert \leq 2C_5 \theta^{-n/2-1/2}$;
\item $\vert \partial_t \hat{\psi} \vert + \vert \Delta \hat{\psi} \vert \leq 2C_5 \theta^{-n/2-1}$.
\end{enumerate}
The first inequality follows from \eqref{eq:cutoff1} and $C_2<C_5$.
The second one is a consequence of $x_0\in \Omega$.
By \eqref{eq:cutoff1} and Theorem \ref{thm:Zhang},
on $M\times [t_0-\tau,t_0]$,
\begin{equation*}
|\nabla G|(x,t)  \leq \frac{1}{\sqrt{t-(t_0-\theta/2)}}\sqrt{\log \frac{C_2 \theta^{-n/2}}{G(x,t)}} G(x,t)\leq 2C_2 \theta^{-n/2-1/2}<2C_5 \theta^{-n/2-1/2},
\end{equation*}
and this implies the third one.
The fourth one can be derived from \eqref{eq:cutoff4}.
If we further define $\tilde{\phi}:=32^{-1} C^{-1}_5 \theta^{n/2+1}\hat{\psi}$,
then it satisfies the following properties:
\begin{enumerate}\setlength{\itemsep}{+0.7mm}
\item $0 \leq \tilde{\phi} < 1$;
\item $\tilde{\phi}(x_0,t_0)\geq C_6 \,\theta$;
\item $ \vert \nabla \tilde{\phi} \vert \leq 1/8$;
\item $\vert \partial_t \tilde{\phi} \vert + \vert \Delta \tilde{\phi} \vert \leq 1/8$.
\end{enumerate}
Once we obtain this function,
we can conclude the desired assertion by the same argument as in the proof of \cite[Theorem 1.3]{BZ1}.
We complete the proof.
\end{proof}

%%%%%%%%%%%%%%%%%%%%%%%%%%%%
%%%%%%%%%%%%%%%%%%%%%%%%%%%%
%%%%%%%%%%%%%%%%%%%%%%%%%%%%
\section{Mean value inequality}\label{sec:mean value}

In the present section,
we produce a mean value inequality for conjugate heat equation.
To do so,
we first yield the following integral estimate based on the Moser iteration argument,
which has been obtained by Bamler-Zhang \cite{BZ1} for Ricci flow (see \cite[Lemma 4.1]{BZ1}):
\begin{lem}\label{lem:premean}
Let $p\geq 2$.
Then there are positive constants $\beta \in (0,1)$ and $C>0$ depending only on $n,T,g(0)$ and $p$ such that the following holds:
For $t_0 \in (0,T)$ and $r_0 \in (0, \sqrt{t_0}]$,
we assume $H\leq r^{-2}_0$.
Let $u \in C^\infty ( M \times [t_0, t_0 + r_0^2])$ denote a positive solution to the conjugate heat equation.
Then for all $x_0 \in M$ we have
\begin{equation*}
\left( \int_{Q^+ (x_0,\beta r_0, t_0)} u^p\, dm\, dt \right)^{1/p}\leq \frac{C}{r_0^{(n+2) (p-2)/2p}} \left(\int _{Q^+ (x_0, r_0,  t_0 )} u^2\, dm \,dt \right)^{1/2}.
\end{equation*}
\end{lem}
\begin{proof}
By the rescaling,
we may assume $r_0 = 1$ (see Remark \ref{rem:rescaling}).
It holds that $t_0\geq 1$ and $H\leq 1$.
Moreover,
$H \geq -n$ on $M \times [t_0 -1/2, t_0+1]$ in view of Proposition \ref{prop:lowerS}.
Let $\alpha, \rho \in (0,1)$ be the constants obtained in Theorems \ref{thm:distcon} and \ref{thm:cutoff},
respectively.
We define
\begin{equation*}
\theta:=\alpha \in (0,1).
\end{equation*}
Note that $(\rho \theta)^2  \in (0,\alpha)$.
By Theorem \ref{thm:distcon} we have
\begin{equation}\label{eq:premean3}
B(x_0,\theta,t_1)\subset B(x_0,1,t_2)
\end{equation}
for all $t_1,t_2 \in [t_0,t_0+(\rho\theta)^2]$.
We further set
\begin{equation*}
\sigma:=\frac{\alpha \rho \theta}{\sqrt{2}} \in (0,\rho \theta).
\end{equation*}
Notice that
$2\sigma^2 \in (0, \alpha(\rho \theta)^2)$.
Using Theorem \ref{thm:distcon} again,
we obtain
\begin{equation}\label{eq:premean4}
B(x_0,\sigma,t_1)\subset B(x_0,\rho \theta,t_2)
\end{equation}
for all $t_1,t_2 \in [t_0,t_0+2\sigma^2]$.
With the help of Theorem \ref{thm:cutoff},
there is $\phi \in C^\infty (M \times [t_0, t_0 + 2\sigma^2])$ such that the following holds:
\begin{enumerate}\setlength{\itemsep}{+1.0mm}
\item $0 \leq \phi < 1$;
\item $\phi \geq \rho$ on $B' \times [t_0,t_0+2\sigma^2]$;
\item $\phi = 0$ on $(M \setminus B) \times [t_0, t_0 + 2 \sigma^2 ]$;
\item $|\nabla \phi | \leq \theta^{-1}$ and $|\partial_t \phi | \leq \theta^{-2}$,
\end{enumerate}
where $B := B(x_0,\theta,t_0 + 2 \sigma^2)$ and $B':=B(x_0, \rho \theta,t_0+2\sigma^2)$.
By \eqref{eq:premean3} and \eqref{eq:premean4},
we possess
\begin{equation}\label{eq:premean1}
Q^+ (x_0 ,\sigma,t_0  ) \subset B' \times [t_0,t_0+2\sigma^2] \subset B \times [t_0,t_0+2\sigma^2] \subset Q^+ (x_0 ,1,t_0  ).
\end{equation}
Let $\eta \in C^\infty ([t_0, t_0 + 2\sigma^2])$ be a function satisfying the following properties:
\begin{enumerate}\setlength{\itemsep}{+1.0mm}
\item $0\leq \eta \leq 1$;
\item $\eta \equiv 1$ on $[t_0, t_0 + \sigma^2]$;
\item $\eta (t_0 + 2 \sigma^2) = 0$;
\item $|\eta' | \leq C_1$.
\end{enumerate}
Then we define $\psi (x,t) := \eta (t) \phi (x,t)$.
Note that $\psi \geq \rho$ on $Q^+ (x_0, \sigma,t_0)$.

Let $p\geq 2$.
Since $u$ is a solution to the conjugate heat equation,
we see
\begin{equation*}
\partial_t u^{p/2}+ \Delta u^{p/2} - \frac{p}2 H u^{p/2}=\frac{p}{2}\left(\frac{p}{2}-1 \right)u^{\frac{p}{2}-2}|\nabla u|^2 \geq 0.
\end{equation*}
On $[t_0, t_0 + 2 \sigma^2]$,
it holds that
\begin{align*}
 \frac{d}{dt} \int_{M}\, u^p \psi^2\, dm &= 2 \int_{M} (\partial_t u^{p/2}) u^{p/2} \psi^2 \,dm + 2 \int_{M} u^p (\partial_t \psi) \psi \,dm -  \int_{M} \,u^p \psi^2 H\, dm \\
 &\geq 2 \int_{M} \left(- \Delta u^{p/2} + \frac{p}2 H u^{p/2} \right) u^{p/2} \psi^2 dm - C_2\int_{B} u^p \,dm\\
 &\geq 2 \int_{M} \,\langle \nabla u^{p/2}, \psi\, \nabla (u^{p/2} \psi) + (u^{p/2} \psi) \nabla \psi \rangle\, dm - C_3  \int_{B} u^p \,dm \\
 &= 2 \int_{M} \langle \nabla (u^{p/2} \psi) - u^{p/2} \nabla \psi,\nabla (u^{p/2} \psi)  + u^{p/2} \nabla \psi \rangle \,dm - C_3  \int_{B} u^p \,dm \\
  &= 2 \int_{M} \big|\nabla (u^{p/2} \psi ) \big|^2 dm - 2 \int_{M} \, u^{p} |\nabla \psi|^2 dm - C_3 \int_{B} u^p \,dm \\
  &\geq  2 \int_{M} \big|\nabla (u^{p/2} \psi ) \big|^2 \,dm - C_3 \int_{B} u^p \,dm.
 \end{align*}
From Proposition \ref{prop:Sob},
we conclude
\begin{equation*}
\frac{d}{dt} \int_{M} u^p \psi^2 \,dm \geq C_4 \bigg( \int_{M}\, u^{\frac{p n}{n-2}} \psi^{\frac{2n}{n-2}} \,dm \bigg)^{\frac{n-2}{n}} - C_5 \int_{B}\, u^p \,dm.
\end{equation*}
Let us integrate this inequality from a fixed $t_1 \in [t_0, t_0 + 2 \sigma^2]$ to $t_0 + 2 \sigma^2$.
We deduce
\begin{align*}
&\quad\,\,-\int_{B'} \,u^p(\cdot, t_1)\psi^2(\cdot, t_1)\, dm_{t_1}  \\
&\geq    - \int_{M} u^p(\cdot, t_1) \psi^2(\cdot, t_1) \,dm_{t_1}  \\
  &\geq C_4 \int_{t_1}^{t_0 + 2\sigma^2}  \bigg( \int_{M} u^{\frac{pn}{n-2}} \psi^{\frac{2n}{n-2}} \,dm \bigg)^{\frac{n-2}n} dt - C_5 \int_{t_1}^{t_0 +  2 \sigma^2} \int_B\, u^p \,dm\, dt.
\end{align*}
This leads us to
\begin{align*}
 \sup_{t \in [t_0, t_0 + 2\sigma^2]} \int_{B'} u^p (\cdot, t)\psi^2(\cdot, t)\, dm_t &\leq C_6\,\int_{t_0}^{t_0 +  2 \sigma^2} \int_B\, u^p \,dm\, dt,\\
 \int_{t_0}^{t_0 + 2\sigma^2}  \bigg( \int_{B'} u^{\frac{pn}{n-2}}\psi^{\frac{2n}{n-2}}  \,dm \bigg)^{\frac{n-2}n} dt &\leq C_7\int_{t_0}^{t_0 +  2 \sigma^2} \int_B\, u^p \,dm\, dt.
\end{align*}
By the H\"older inequality,
\begin{align*}
&\quad\,\, \int_{t_0}^{t_0 + 2 \sigma^2}\,  \int_{B'} u^{p\left( 1 + \frac{2}n\right)}\psi^{2\left( 1 + \frac{2}n\right)} \,dm\, dt\\
&= \int_{t_0}^{t_0 + 2 \sigma^2}\,  \int_{B'} \,\left(u^{p}\psi^2\right)\left(u^{\frac{2p}{n}}\psi^\frac{4}{n}\right)\, dm\, dt\\
&\leq  \int_{t_0}^{t_0 + 2 \sigma^2} \bigg( \int_{B'} u^{\frac{pn}{n-2}} \psi^{\frac{2n}{n-2}} \,dm \bigg)^{\frac{n-2}{n}} \bigg( \int_{B'} u^p \psi^2 \,dm \bigg)^{\frac{2}n}\, dt \\
&\leq  \left[ \int_{t_0}^{t_0 + 2 \sigma^2} \bigg( \int_{B'} u^{\frac{pn}{n-2}}\psi^{\frac{2n}{n-2}}\, dm \bigg)^{\frac{n-2}{n}} dt \right] \left[ C_6  \int_{t_0}^{t_0+ 2\sigma^2} \int_Bu^p dm\, dt \right]^{\frac{2}n}  \\
& \leq  \bigg( C_8 \int_{t_0}^{t_0+ 2 \sigma^2} \int_B u^p \,dm \,dt \bigg)^{1 + \frac{2}n}.
\end{align*}
In view of \eqref{eq:premean1} and $\psi \geq \rho$ on $Q^+ (x_0, \sigma,t_0)$,
we arrive at
\begin{equation*}
\bigg( \int_{Q^+ (x_0, \sigma,t_0 )} u^{p \left( 1 + \frac{2}n \right)}\, dm\, dt \bigg)^{1 / p\left( 1 + \frac{2}n \right)} \leq C_9 \bigg( \int_{Q^+ (x_0, 1,t_0)} u^p \,dm \,dt \bigg)^{1/p}.
\end{equation*}
Once we obtain this estimate,
we can conclude the desired one by the same argument as in the proof of \cite[Lemma 4.2]{BZ1}.
\end{proof}

We give a proof of the following mean value inequality,
which has been formulated by Bamler-Zhang \cite{BZ1} for Ricci flow (see \cite[Lemma 4.2]{BZ1}):
\begin{thm}\label{thm:mean}
There exist positive constants $\gamma\in (0,1)$ and $C > 0$ depending only on $n,T$ and $g(0)$ such that the following holds:
For $t_0 \in (0,T)$ and $r_0 \in (0, \sqrt{t_0}]$,
we assume $H\leq r^{-2}_0$.
Let $u \in C^\infty ( M \times [t_0, t_0 + r_0^2])$ be a positive solution to the conjugate heat equation.
Then for all $x_0\in M$ we have
\begin{equation*}
\sup_{Q^+(x_0, \gamma r_0,t_0)} u^2 \leq \frac{C}{r_0^{n+2}} \,\int_{Q^+(x_0, r_0,t_0)} u^2 \,dm\, dt.
\end{equation*}
\end{thm}
\begin{proof}
By the rescaling,
we may assume $r_0 = 1$ (see Remark \ref{rem:rescaling}).
We have $t_0\geq 1$ and $H\leq 1$.
Furthermore,
we have $H \geq -n$ on $M \times [t_0 -1/2, t_0+1]$ by Proposition \ref{prop:lowerS}.
Let $\alpha,\rho \in (0,1)$ be the constants obtained in Theorems \ref{thm:distcon} and \ref{thm:cutoff},
respectively.
Let $\beta \in (0,1)$ be the constant obtained in Lemma \ref{lem:premean} for a fixed $p>n+2$.
We define
\begin{equation*}
\theta:=\alpha \beta \in (0,1).
\end{equation*}
Note that $(\rho \theta)^2\in (0,\alpha \beta^2)$.
By Theorem \ref{thm:distcon},
it holds that
\begin{equation}\label{eq:mean1000}
B(x_0,\theta,t_1)\subset B(x_0,\beta,t_2)
\end{equation}
for all $t_1,t_2\in [t_0,t_0+(\rho\theta)^2]$.
We further define
\begin{equation*}
\gamma:=\frac{\alpha \rho \theta}{\sqrt{2}} \in (0,\rho \theta).
\end{equation*}
Notice that
$2 \gamma^2\in (0,\alpha (\rho \theta)^2)$.
In virtue of Theorem \ref{thm:distcon},
we possess
\begin{equation}\label{eq:mean10000}
B(x_0,\gamma,t_1)\subset B(x_0,\rho \theta,t_2)
\end{equation}
for all $t_1,t_2\in [t_0,t_0+2\gamma^2]$.
By Theorem \ref{thm:cutoff},
there is $\phi \in C^\infty (M \times [t_0, t_0 + 2\gamma^2])$ such that the following holds:
\begin{enumerate}\setlength{\itemsep}{+1.0mm}
\item $0 \leq \phi < 1$;
\item $\phi \geq \rho$ on $B' \times [t_0,t_0+2\gamma^2]$;
\item $\phi = 0$ on $(M \setminus B) \times [t_0, t_0 + 2 \gamma^2 ]$;
\item $|\nabla \phi | \leq \theta^{-1}$ and $|\partial_t \phi |+|\Delta \phi| \leq \theta^{-2}$,
\end{enumerate}
where $B:=B(x_0,\theta,t_0+2\gamma^2)$ and $B':=B(x_0,\rho\,\theta,t_0+2\gamma^2)$.
From \eqref{eq:mean1000} and \eqref{eq:mean10000}, we deduce
\begin{align}\label{eq:mean0}
 Q^+ (x_0, \gamma,t_0) \subset B' \times [t_0,t_0+2\gamma^2]\subset B \times [t_0,t_0+2\gamma^2] \subset Q^+ (x_0, \beta,t_0 ).
\end{align}
Let $\eta \in C^\infty ([t_0, t_0 + 2\gamma^2])$ be a function satisfying the following properties:
\begin{enumerate}\setlength{\itemsep}{+1.0mm}
\item $0\leq \eta \leq 1$;
\item $\eta \equiv 1$ on $[t_0, t_0 + \gamma^2]$;
\item $\eta (t_0 + 2 \gamma^2) = 0$;
\item $|\eta' | \leq C_1$.
\end{enumerate}
Then we define $\psi (x,t) := \eta (t) \phi (x,t)$.
In view of \eqref{eq:mean0},
we notice $\psi \geq \rho$ on $Q^+ (x_0, \gamma,t_0)$.
For a fixed $(x,t) \in Q^+ (x_0, \gamma,t_0)$, we set $G(z, l):=G(z,l;x,t)$.
Then
\begin{equation*}
 (u \psi) (x,t) = -\int_t^{t_0 + 2\gamma^2} \int_{M} G \big( u \Delta \psi + u \partial_l \psi + 2 \langle \nabla u, \nabla \psi\rangle \big) dm\, dl.
 \end{equation*}
 Here we used
 \begin{equation*}
\partial_l (u\psi)+\Delta (u \psi) - H (u \psi) = u \Delta \psi + u \partial_l \psi + 2 \langle \nabla u, \nabla \psi \rangle,
\end{equation*}
which is a consequence of the fact that $u$ is a solution to the conjugate heat equation.
Integration by parts yields
\begin{equation*}
 (u \psi) (x,t) =  - \int_t^{t_0 + 2\gamma^2} \int_{M} G u ( - \Delta \psi + \partial_l \psi)+2u \langle \nabla G, \nabla \psi\rangle\,  dm\, dl.
\end{equation*}
From the properties of $\psi$ and \eqref{eq:mean0}, it follows that
\begin{equation}\label{eq:mean3}
(u \psi)(x,t) \leq C_2 \int_{t}^{t_0 + 2\gamma^2} \int_{B} \,\left(Gu+u|\nabla G|\right) \,dm\, dl=:C_2(I_1+I_2).
\end{equation}

We first estimate $I_1$.
By the H\"older inequality, \eqref{eq:mean0} and Lemma \ref{lem:premean},
\begin{align}\label{eq:mean10}
I_1 &\leq \bigg( \int_{t}^{t_0 + 2\gamma^2} \int_{B} G^q \,dm\, dl \bigg)^{1/q} \bigg( \int_{t}^{t_0 + 2\gamma^2} \int_{B} u^p \,dm\, dl \bigg)^{1/p} \\ \notag
     &\leq \bigg( \int_{t}^{t_0 + 2\gamma^2} \int_{B} G^q \,dm\, dl \bigg)^{1/q} \bigg( \int_{Q^+(x_0,\beta, t_0)} u^p \,dm\, dl \bigg)^{1/p}\\ \notag
     &\leq C_3 \bigg( \int_{t}^{t_0 + 2\gamma^2} \int_{B} G^q\, dm\, dl \bigg)^{1/q} \Vert u \Vert_{L^2 (Q^+(x_0,1, t_0))}
\end{align}
for $q = p/(p-1)$.
By Proposition \ref{prop:upper heat kernel},
for all $l\in (t,t_0+2\gamma^2]$ we have
\begin{equation}\label{eq:mean1}
G(z,l) \leq \frac{C_4}{(l-t)^{n/2}}.
\end{equation}
We also see
\begin{equation*}
\frac{d}{dl} \int_{M}\, G \,dm = - \int_{M} \,H\,G \,dm \leq n \int_{M} G \,dm
\end{equation*}
since $H \geq -n$ on $M \times [t_0 -1/2, t_0+1]$;
in particular,
for all $l\in (t,t_0+2\gamma^2]$,
\begin{equation}\label{eq:mean2}
\int_{M} G \,dm \leq \exp(n(l-t)) \leq C_{5}.
\end{equation}
Combining \eqref{eq:mean10}, \eqref{eq:mean1} and \eqref{eq:mean2},
we obtain
\begin{align*}
I_1 &\leq  C_6 \bigg( \int_{t}^{t_0 + 2\gamma^2} \int_{B} G  \frac1{(l-t)^{(q-1) n/2}} \,dm \,dl \bigg)^{1/q}  \Vert u \Vert_{L^2 (Q^+(x_0, 1,t_0))} \\
      &\leq C_7 \bigg( \int_{t}^{t_0+ 2\gamma^2}  \frac1{(l-t)^{(q-1) n/2}} \,dl \bigg)^{1/q} \Vert u \Vert_{L^2 (Q^+(x_0, 1,t_0))} .
\end{align*}
Since $p > n+2$, we have $(q - 1) n/2 < 1$;
in particular,
\begin{equation}\label{eq:mean4}
I_1 \leq C_8 \Vert u \Vert_{L^2(Q^+(x_0, 1,t_0))}.
\end{equation}

We next derive a bound of $I_2$.
The estimate \eqref{eq:mean1} implies that for each $l \in (t, t_0+2\gamma^2]$,
\begin{equation*}
A_l := \sup_{M \times \left[\frac{l+t}2, l \right]} G \leq \frac{C_9}{(l-t)^{n/2}}.
\end{equation*}
Using Theorem \ref{thm:Zhang},
we possess
\begin{equation*}
\frac{|\nabla G|^2}{G^2}(z,l) \leq \frac{2}{l-t} \log \frac{A_l}{G(z,l)}.
\end{equation*}
Since $G(z,l) / A_l \leq 1$,
we have
\begin{equation*}
|\nabla G |^q \leq G \frac{C_{10} A_l^{q-1}}{(l-t)^{q/2}} \bigg( \frac{G}{A_l} \bigg)^{q-1} \bigg( {\log \frac{A_l}{G}} \bigg)^{q/2}\leq G \frac{C_{11} A_l^{q-1}}{(l-t)^{q/2}}\leq \frac{C_{12}}{(l-t)^{((n+1)q - n)/2}} G
\end{equation*}
at $(z,l)$,
where we used a fact that a function $\varphi(\xi):=\xi^{q-1} (\log (1/\xi))^{q/2}$ takes its maximum value $e^{-q/2} (q/2(q-1))^{q/2}$ at $\xi=e^{-q/2(q-1)}$ on $(0,1]$.
Using the H\"older inequality, this gradient bound, \eqref{eq:mean2}, Lemma \ref{lem:premean} with \eqref{eq:mean0}, we conclude
\begin{align*}
I_2 &\leq \bigg( \int_{t}^{t_0 + 2\gamma^2} \int_{B} |\nabla G|^q\, dm\, dl \bigg)^{1/q} \bigg( \int_{t}^{t_0 +2\gamma^2} \int_{B} u^p \,dm\, dl \bigg)^{1/p} \\
     &\leq \bigg( \int_{t}^{t_0 + 2\gamma^2} \int_{B} |\nabla G|^q\, dm \,dl \bigg)^{1/q} \bigg( \int_{Q^+(x_0,\beta, t_0)} u^p \,dm\, dl \bigg)^{1/p} \\
     &\leq C_{13}\,\bigg( \int_{t}^{t_0 + 2\gamma^2}  \frac{1}{(l-t)^{((n+1)q - n)/2}}\, dl \bigg)^{1/q} \Vert u \Vert_{L^2 (Q^+ (x_0, 1,t_0))}.
\end{align*}
From $p > n+2$, we deduce $((n+1)q - n)/2 < 1$, and hence
\begin{equation}\label{eq:mean5}
I_2 \leq C_{14} \Vert u \Vert_{L^2 (Q^+ (x_0,1, t_0))}.
\end{equation}
Combining \eqref{eq:mean3}, \eqref{eq:mean4}, \eqref{eq:mean5},
we arrive at
\begin{equation*}
(u\psi)(x,t) \leq C_{15} \Vert u \Vert_{L^2 (Q^+ (x_0, 1,t_0))}.
\end{equation*}
The desired assertion follows from $\psi(x,t)\geq \rho$.
We complete the proof.
\end{proof}

\begin{rem}
Besides the mean value inequality for conjugate heat equation,
Bamler-Zhang \cite{BZ1} have also produced a mean value inequality for heat equation (see \cite[Lemma 4.3, Proposition 4.4]{BZ1}).
Theorems \ref{thm:distcon} and \ref{thm:cutoff} together with their argument enable us to generalize such a mean value inequality in our setting.
\end{rem}

%%%%%%%%%%%%%%%%%%%%%%%%%%%%
%%%%%%%%%%%%%%%%%%%%%%%%%%%%
%%%%%%%%%%%%%%%%%%%%%%%%%%%%
\section{Gaussian heat kernel estimate}\label{sec:heat kernel bounds}
In this section,
we give a proof of Theorem \ref{thm:mainhk}.
First,
we prove the lower bound \eqref{thm:mainhklow}.
To do so,
we prepare the following lemma:
\begin{lem}\label{lem:lowerlem}
For any $A>0$,
there is a constant $B>1$ depending only on $n,T,g(0)$ and $A$ such that the following holds:
We assume $H \leq H_1$ for $H_1> 0$.
For $s,t \in [0,T)$ with $s<t$,
we suppose $t-s \leq A\, H_1^{-1}$ and $s \geq (t-s)/A$.
If $d_t(x,y)\geq B\sqrt{t-s}$,
then $d_s(x,y)\geq B^{-1}d_t(x,y)$.
\end{lem}
\begin{proof}
Let $\alpha \in (0,1)$ be the constant obtained in Theorem \ref{thm:distcon}.
We first consider the case of $A\in (0,\alpha]$.
We set
\begin{equation*}
\theta:=\alpha^{-1} A \in (0,1],\quad B:=\alpha^{-1}A^{-1/2},\quad r_0:=A^{-1/2}\sqrt{t-s}.
\end{equation*}
It holds that
\begin{equation*}
r_0\leq \sqrt{s}\leq \sqrt{t},\quad H\leq H_1 \leq  \frac{A}{t-s}= r^{-2}_0,\quad t-s\leq \theta^{-1}(t-s)\leq \alpha r^{2}_0,\quad d_t(x,y)\geq \alpha^{-1}r_0\geq r_0,
\end{equation*}
and hence one can apply Theorem \ref{thm:distcon}.
It follows that
\begin{equation*}
d_s(x,y)\geq \alpha d_t(x,y)\geq \alpha A^{1/2}d_t(x,y)=B^{-1}d_t(x,y),
\end{equation*}
which is a desired estimate.

We next consider the remaining case of $A\in (\alpha,\infty)$.
We put
\begin{equation*}
\theta:=\alpha^{-1} A \in (1,\infty),\quad B:=\alpha^{-(\theta+2)},\quad r_0:=A^{-1/2}\sqrt{t-s}.
\end{equation*}
Let $N\in \mathbb{N}$ be the integer part of $\theta$,
and let $s=\tau_{N+1}\leq \tau_{N}< \dots< \tau_1 < \tau_0=t$ be a division of $[s,t]$ such that $\tau_i-\tau_{i+1}=\theta^{-1}(t-s)$ for $i=0,\dots,N-1$.
We see
$\tau_N-\tau_{N+1}\in [0, \theta^{-1}(t-s))$.
We possess
\begin{equation*}
r_0\leq \sqrt{s} \leq \sqrt{\tau_i}\leq \sqrt{t},\quad H\leq H_1 \leq  \frac{A}{t-s}= r^{-2}_0,\quad \tau_i-\tau_{i+1}\leq \theta^{-1}(t-s)= \alpha r^{2}_0
\end{equation*}
for all $i=0,\dots,N$.
In view of
\begin{equation*}
B\geq \alpha^{-\theta-1}A^{-1/2}\geq \alpha^{-N-1}A^{-1/2},
\end{equation*}
we see
\begin{equation*}
d_{\tau_0}(x,y)=d_{t}(x,y)\geq B\sqrt{t-s}\geq \alpha^{-N-1} A^{-1/2}\sqrt{t-s}=\alpha^{-N-1}r_0 \geq r_0,
\end{equation*}
and hence Theorem \ref{thm:distcon} implies
\begin{equation*}
d_{\tau_1}(x,y)\geq \alpha \,d_{\tau_0}(x,y)\geq \alpha^{-N}r_0\geq r_0.
\end{equation*}
Using Theorem \ref{thm:distcon} again,
we obtain
\begin{equation*}
d_{\tau_2}(x,y)\geq \alpha \,d_{\tau_1}(x,y)\geq \alpha^{-N+1}r_0\geq r_0.
\end{equation*}
By repeating this procedure,
we arrive at
\begin{equation*}
d_{\tau_{i}}(x,y)\geq \alpha d_{\tau_{i-1}}(x,y)\geq \alpha^{-N+(i-1)}r_0\geq r_0
\end{equation*}
for all $i=1,\dots,N+1$.
Therefore,
\begin{align*}
d_s(x,y)&=d_{\tau_{N+1}}(x,y)\\
&\geq \alpha d_{\tau_{N}}(x,y)\geq \alpha^2 d_{\tau_{N-1}}(x,y)\geq \cdots \geq \alpha^{N+1}d_{t}(x,y)\\
&\geq \alpha^{\theta+2}d_t(x,y)\geq B^{-1}\sqrt{t-s}.
\end{align*}
This proves the lemma.
\end{proof}

Let us show the lower bound \eqref{thm:mainhklow}.
\begin{prop}\label{prop:mainlower} 
For any $A>0$,
there exist positive constants $\mathcal{C}_1,\mathcal{C}_2>0$ depending only on $n,T,g(0)$ and $A$ such that the following holds:
We assume $H \leq H_1$ for $H_1> 0$.
For $s,t \in [0,T)$ with $s<t$,
we suppose $t-s \leq A\, H_1^{-1}$ and $s \geq (t-s)/A$.
Then we have
\begin{equation*}
G(x,t;y,s) \geq \frac{\mathcal{C}_1}{(t-s)^{n/2}} \exp \left( {- \frac{\mathcal{C}_2\, d^2_s (x, y)}{t-s}} \right).
\end{equation*} 
\end{prop}
\begin{proof}
We possess
\begin{equation}\label{eq:hoge1}
H\leq H_1 \leq  \frac{A}{t-s}.
\end{equation}
Furthermore,
in view of Lemma \ref{prop:lowerS},
\begin{equation}\label{eq:hoge2}
-\frac{nA}{2(t-s)}\leq-\frac{n}{2s}\leq H
\end{equation}
on $M \times [s,t]$.
Theorem \ref{thm:lowerGauss} together with \eqref{eq:hoge1} tells us that
\begin{equation}\label{eq:weakGauss}
G(x,t; y,s) \geq \frac{C}{(t-s)^{n/2}} \exp \left( { - \frac{d^2_t(x,y)}{t-s}} \right).
\end{equation}
Let $B>1$ be the constant obtained in Lemma \ref{lem:lowerlem}.
If $d_t (x,y) \geq B\sqrt{t-s}$,
then $d_s (x,y) \geq B^{-1}\,d_t (x,y)$.
From \eqref{eq:weakGauss} it follows that
\begin{equation*}
G(x,t; y,s) \geq \frac{C}{(t-s)^{n/2}} \exp \left( {- \frac{d^2_t(x, y)}{t-s}  } \right) \geq \frac{C}{(t-s)^{n/2}} \exp \left( {- \frac{B^2 d^2_s (x, y)}{t-s} } \right).
\end{equation*}
On the other hand,
if $d_t (x,y)< B\sqrt{t-s}$,
then \eqref{eq:weakGauss} also leads to
\begin{equation*}
G(x,t; y,s) \geq \frac{C}{(t-s)^{n/2}} \exp \left( {- \frac{ d^2_t (x, y)}{t-s} } \right) \geq \frac{Ce^{-B^2} }{(t-s)^{n/2}}  \geq \frac{Ce^{-B^2} }{(t-s)^{n/2}} \exp \left( {- \frac{d^2_s (x, y)}{t-s} }\right).
\end{equation*}
Thus,
we complete the proof.
\end{proof}

We next prove the upper bound \eqref{thm:mainhkup}.
Similarly to Lemma \ref{lem:lowerlem},
we prepare the following:
\begin{lem}\label{lem:upperlem}
For any $A>0$,
there is a constant $B>1$ with $B^{-1}A^{-1/2}< 1$ depending only on $n,T,g(0)$ and $A$ such that the following holds:
Assume $H \leq H_1$ for $H_1> 0$.
For $s,t \in [0,T)$ with $s<t$,
we suppose $t-s \leq A\, H_1^{-1}$ and $s \geq (t-s)/A$.
Assume $d_s(x,y)\geq B\sqrt{t-s}$.
Then for all $l_1,l_2\in [s,t]$ we have
\begin{equation*}
B^{-1}\,d_{l_1}(x,y)\leq d_{l_2}(x,y)\leq B\,d_{l_1}(x,y).
\end{equation*}
\end{lem}
\begin{proof}
Let $\alpha \in (0,1)$ be the constant obtained in Theorem \ref{thm:distcon}.
We begin with the case of $A\in (0,\alpha]$.
Similarly to the proof of Lemma \ref{lem:lowerlem},
we put
\begin{equation*}
\theta:=\alpha^{-1} A \in (0,1],\quad B:=\alpha^{-1}A^{-1/2},\quad r_0:=A^{-1/2}\sqrt{t-s}.
\end{equation*}
Note that $B>1$ and $B^{-1}A^{-1/2}< 1$.
For all $l_1,l_2\in [s,t]$ we see
\begin{align*}
r_0&\leq \sqrt{s}\leq \sqrt{t},\quad H\leq H_1 \leq  \frac{A}{t-s}= r^{-2}_0,\\
|l_1-l_2|&\leq \theta^{-1}(t-s)\leq \alpha r^{2}_0,\quad d_{s}(x,y)\geq \alpha^{-1}r_0\geq r_0.
\end{align*}
In virtue of Theorem \ref{thm:distcon},
we obtain
\begin{equation*}
d_{l_1}(x,y)\geq \alpha d_{s}(x,y)\geq r_0,\quad d_{l_2}(x,y)\geq \alpha d_{l_1}(x,y)\geq \alpha A^{1/2}d_{l_1}(x,y)= B^{-1}d_{l_1}(x,y).
\end{equation*}
By switching the roles of $l_1$ and $l_2$,
we arrive at the desired estimate.

Next,
we study the case of $A\in (\alpha,\infty)$.
As in the proof of Lemma \ref{lem:lowerlem},
we define
\begin{equation*}
\theta:=\alpha^{-1} A \in (1,\infty),\quad B:=\alpha^{-(\theta+2)},\quad r_0:=A^{-1/2}\sqrt{t-s},
\end{equation*}
where $B>1$ and $B^{-1}A^{-1/2}< 1$.
Let $N\in \mathbb{N}$ be the integer part of $\theta$,
and let $s=\tau_{0}< \tau_{1}< \dots< \tau_N \leq \tau_{N+1}=t$ be a division of $[s,t]$ such that $\tau_{i+1}-\tau_{i}=\theta^{-1}(t-s)$ for $i=0,\dots,N-1$.
We see
$\tau_{N+1}-\tau_{N}\in [0, \theta^{-1}(t-s))$.
It holds that
\begin{equation*}
r_0\leq \sqrt{s} \leq \sqrt{\tau_i}\leq \sqrt{t},\quad H\leq H_1 \leq  \frac{A}{t-s}= r^{-2}_0,\quad \tau_{i+1}-\tau_{i}\leq \theta^{-1}(t-s)= \alpha r^{2}_0
\end{equation*}
for every $i=0,\dots,N$.
Since
\begin{equation*}
d_{\tau_0}(x,y)=d_{s}(x,y)\geq B\sqrt{t-s}\geq \alpha^{-N-1} A^{-1/2}\sqrt{t-s}=\alpha^{-N-1}r_0 \geq r_0,
\end{equation*}
Theorem \ref{thm:distcon} tells us that
\begin{equation*}
d_{\tau_1}(x,y)\geq \alpha \,d_{\tau_0}(x,y)\geq \alpha^{-N}r_0\geq r_0.
\end{equation*}
In the same manner as in the proof of Lemma \ref{lem:lowerlem},
we use Theorem \ref{thm:distcon} repeatedly.
We obtain
\begin{equation}\label{eq:upperlem1}
d_{\tau_{i}}(x,y)\geq \alpha d_{\tau_{i-1}}(x,y)\geq \alpha^{-N+(i-1)}r_0\geq r_0
\end{equation}
for all $i=1,\dots,N+1$.
For $l_1,l_2\in [s,t]$,
let $i_0\in \{0,\dots,N\}$ satisfy $l_1\in [\tau_{i_0},\tau_{i_0+1}]$.
Theorem \ref{thm:distcon} together with \eqref{eq:upperlem1} leads us to
\begin{equation}\label{eq:upperlem2}
d_{l_1}(x,y)\geq \alpha d_{\tau_{i_0}}(x,y)\geq \alpha^{-N+i_0} r_0\geq r_0.
\end{equation}
By using Theorem \ref{thm:distcon} together with \eqref{eq:upperlem1}, \eqref{eq:upperlem2} at most $N+1$ times,
we conclude
\begin{align*}
d_{l_2}(x,y)\geq \alpha^{N+1} d_{l_1}(x,y)\geq B^{-1}d_{l_1}(x,y).
\end{align*}
Switching the roles of $l_1$ and $l_2$,
we complete the proof.
\end{proof}

We now prove the upper bound \eqref{thm:mainhkup}.
\begin{prop}\label{prop:mainupper} 
For any $A>0$,
there exist positive constants $\mathcal{C}_3,\mathcal{C}_4>0$ depending only on $n,T,g(0)$ and $A$ such that the following holds:
We assume $H \leq H_1$ for $H_1> 0$.
For $s,t \in [0,T)$ with $s<t$,
we also suppose $t-s \leq A\, H_1^{-1}$ and $s \geq (t-s)/A$.
Then we have
\begin{equation*}
G(x, t; y, s) \leq \frac{\mathcal{C}_3}{(t-s)^{n/2}} \exp \left({- \frac{\mathcal{C}_4\,d^2_s (x, y)}{t-s}} \right).
\end{equation*} 
\end{prop}
\begin{proof}
We again notice \eqref{eq:hoge1} and \eqref{eq:hoge2}.
By Proposition \ref{prop:upper heat kernel}, 
for $l \in [s, t)$ we have
\begin{equation}\label{eq:preupper}
G(x, t; y, l) \leq \frac{C_1}{(t-l)^{n/2}}.
\end{equation}
Let $B>1$ be the constant in Lemma \ref{lem:upperlem} satisfying $B^{-1}A^{-1/2}< 1$.
When $d_s (x,y)<B\sqrt{t-s}$,
\eqref{eq:preupper} implies
\begin{equation*}
G(x,t; y, s) \leq \frac{C_1}{(t-s)^{n/2}} \leq \frac{C_1 e^{B^2}}{(t-s)^{n/2}}\exp \left( {- \frac{d^2_s(x,y)}{t-s} } \right).
\end{equation*}
This is a desired one.

We consider the case of $d_s (x,y)\geq B\sqrt{t-s}$.
We set
\begin{equation*}
c:=B^{-1}A^{-1/2}\in (0,1).
\end{equation*}
For all $l \in [s,t)$ we see
\begin{equation*}
c\sqrt{t-l}< \sqrt{T},\quad H\leq \frac{A}{t-s}\leq \frac{AB^2}{t-l}=\frac{1}{c^{2}(t-l)},
\end{equation*}
and hence
\begin{equation}\label{eq:main3}
m(B(x,c\sqrt{t-l},l)) \geq \kappa_1 (t-l)^{n/2}
\end{equation}
by Theorem \ref{thm:noncollapse}.
Proposition \ref{prop:Average} together with \eqref{eq:main3} implies
\begin{align}\label{eq:main1}
&\quad \,\,\frac{1}{ m(B(y,c\sqrt{t-l},l))} \int_{B(y,c\sqrt{t-l},l)} G(x, t; z, l) \,dm(z) \\ \notag
&\leq \bigg( \int_{B(x,c\sqrt{t-l},l)} G(x,t; z, l ) \,dm (z) \bigg)^{-1} \frac{C_2}{(t-l)^{n/2}} \exp \left( - \frac{d^2_{l} (x, y)}{16(t-l)}  \right).
\end{align}
By Proposition \ref{prop:mainlower},
\begin{equation}\label{eq:main2}
G(x,t;z,l) \geq \frac{C_3}{(t-l)^{n/2}} \exp \left( {- \frac{C_4 d^2_l(x, z)}{t-l}} \right)\geq \frac{C_5}{(t-l)^{n/2}}
\end{equation}
for all $z\in B(x,c\sqrt{t-l},l)$.
Combining \eqref{eq:main3}, \eqref{eq:main1}, \eqref{eq:main2},
we conclude
\begin{equation}\label{eq:main4}
\frac{1}{ m(B(y, c\sqrt{t-l},l))} \int_{B(y, c\sqrt{t-l},l)} G(x, t; z, l) \,dm(z) \leq  \frac{C_6}{(t-l)^{n/2}} \exp \left( - \frac{d^2_{l} (x, y)}{16(t-l)}  \right).
\end{equation}
Now,
we possess
\begin{align*}
c\sqrt{t-l}&<A^{-1/2}\sqrt{t-l}\leq A^{-1/2}\sqrt{t-s}\leq \sqrt{s}\leq \sqrt{l},\\
H(\cdot,\tau) &\leq \frac{A}{t-s}\leq \frac{c^2(t-l)}{l-\tau}  \frac{A}{t-s}\leq \frac{c^2A}{l-\tau}
\end{align*}
for $\tau \in [l-c^2(t-l),l]$,
and hence Theorem \ref{thm:inflat} tells us that
\begin{equation}\label{eq:main5}
m_l(B(y,c\sqrt{t-l},l))\leq \kappa_2 (t-l)^{n/2}.
\end{equation}
From \eqref{eq:main4}, \eqref{eq:main5} and Lemma \ref{lem:upperlem}, one can derive
\begin{equation*}
\int_{B(y, c\sqrt{t-l},l)} G(x, t; z, l)\, dm(z) \leq  C_7 \exp \left( {- \frac{d^2_{l}(x,y)}{16(t-l)} } \right) \leq C_7 \exp \left( {- \frac{B^{-2} \,d^2_{s} (x, y)}{16(t-s)}}  \right).
\end{equation*}
This together with \eqref{eq:preupper} leads us to
\begin{equation*}
\int_{B(y, c\sqrt{t-l},l)} G(x, t; z, l)^2\, dm(z) \leq \frac{C_8}{(t-l)^{n/2}} \exp \left({ - \frac{B^{-2}\,d^2_s(x, y)}{16(t-s)}} \right).
\end{equation*}
Integrating over $[s, (s+ t)/2]$, we obtain
\begin{equation*}
\int^{\frac{s+t}{2}}_s \int_{B(y,c\sqrt{t-l},l)} \,G(x, t; z, l)^2 \,dm(z)\, dl \leq \frac{C_9}{(t-s)^{n/2-1}} \exp \left( { -  \frac{B^{-2}\,d^2_s (x,y)}{16(t-s)} } \right).
\end{equation*}
We now set
\begin{equation*}
r:=c\sqrt{\frac{t-s}{2}}.
\end{equation*}
We notice that
$t-l \ge (t-s)/2$ for $l\in [s, (s+t)/2]$,
and hence
\begin{equation*}
Q^{+}(y,r,s)\subset \left\{(z,l)\in M\times [0,T) ~~\middle|~~  z\in B(y,c\sqrt{t-l},l),\, l\in \left[s,\frac{s+t}{2}\right] \right\}.
\end{equation*}
It follows that
\begin{equation*}
 \int_{Q^+ (y, r,s)} G(x, t; z, l)^2\, dm(z)\, dl \leq \frac{C_9}{(t-s)^{n/2-1}} \exp \left( {- \frac{B^{-2}\,d^2_s(x,y)}{16(t-s)} } \right).
 \end{equation*}
By virtue of Theorem \ref{thm:mean},
\begin{equation*}
G (x,t; y, s)^2 \leq \frac{C_{10}}{r^{n+2}} \int_{Q^+ (y, r,s)} G(x, t; z, l)^2 \,dm(z) \,dl \leq \frac{C_{11}}{(t-s)^{n}} \exp \left( {- \frac{B^{-2}\,d^2_s(x, y)}{16(t-s)} } \right).
\end{equation*}
Thus,
we complete the proof.
\end{proof}

We finally prove the gradient estimate \eqref{thm:maingrad}.
\begin{prop}\label{prop:maingrad} 
For any $A>0$,
there exist positive constants $\mathcal{C}_5,\mathcal{C}_6>0$ depending only on $n,T,g(0)$ and $A$ such that the following holds:
We assume $H \leq H_1$ for $H_1> 0$.
For $s,t \in [0,T)$ with $s<t$,
we also suppose $t-s \leq A\, H_1^{-1}$ and $s \geq (t-s)/A$.
Then we have
\begin{equation*}
|\nabla_x G|(x,t;y,s)\leq \frac{\mathcal{C}_5}{(t-s)^{(n+1)/2}} \exp \left({- \frac{ \mathcal{C}_6 \,d^2_s (x, y)}{t-s}} \right).
\end{equation*} 
\end{prop}
\begin{proof}
We set $G(x,t):=G(x,t;y,s)$.
Due to Theorem \ref{thm:Zhang},
\begin{equation*}
\frac{|\nabla G|^2}{G^2}(x,t)  \leq \frac{C_1}{t-s}\log \frac{\mathcal{A}}{G(x,t)},
\end{equation*}
where
\begin{equation*}
\mathcal{A}:=\sup_{M\times \left[\frac{s+t}{2},t  \right]}G.
\end{equation*}
It follows that
\begin{equation*}
|\nabla G|^2(x,t)  \leq \frac{C_1}{t-s} \mathcal{A}G(x,t)\frac{G(x,t)}{\mathcal{A}}\log \frac{\mathcal{A}}{G(x,t)}\leq \frac{C_1}{t-s} \mathcal{A}G(x,t)
\end{equation*}
since $G(x,t)/\mathcal{A}\leq 1$.
By Propositions \ref{prop:upper heat kernel} and \ref{prop:mainupper},
we possess
\begin{equation*}
\mathcal{A}\leq \frac{C_2}{(t-s)^{n/2}},\quad G(x,t)\leq \frac{C_3}{(t-s)^{n/2}} \exp \left({- \frac{C_4\,d^2_s (x, y)}{t-s}} \right).
\end{equation*}
Combining them,
we arrive at
\begin{equation*}
|\nabla G|^2(x,t)  \leq \frac{C_5}{(t-s)^{n+1}} \exp \left({- \frac{C_4\,d^2_s (x, y)}{t-s}} \right).
\end{equation*}
This completes the proof.
\end{proof}

\begin{proof}[Proof of Theorem \ref{thm:mainhk}]
By Propositions \ref{prop:mainlower}, \ref{prop:mainupper}, \ref{prop:maingrad},
we complete the proof.
\end{proof}

%%%%%%%%%%%%%%%%%%%%%%%%%%%%
%%%%%%%%%%%%%%%%%%%%%%%%%%%%
%%%%%%%%%%%%%%%%%%%%%%%%%%%%
\subsection*{{\rm Acknowledgements}}
The authors thank the anonymous referee for useful comments on references.
The first author was supported by JSPS KAKENHI (JP19K14521, 23K03105). 
The second author was supported by JSPS KAKENHI (JP21K20315, 22H04942, 23K12967).

%%%%%%%%%%%%%%%%%%%%%%%%%%%%

\end{document}